\documentclass[11pt,leqno, twoside]{article}
\usepackage{amsthm}
\usepackage{amsmath}
\usepackage{amssymb}
\usepackage{txfonts}
\usepackage{pxfonts}
\usepackage{graphicx}
\usepackage{epstopdf}
\usepackage{epsfig}
\usepackage[all]{xy}
\usepackage{proof}
\usepackage{wrapfig}
\usepackage{sidecap}
\usepackage{subfig}

\newcommand{\N}{\mathbb{N}}

\newcommand{\Q}{\mathbb{Q}}
\newcommand{\R}{\mathbb{R}}

\newtheorem{thm}{Theorem}[section]
\newtheorem*{Thm}{Theorem}
\theoremstyle{definition}
\newtheorem{definition}[thm]{Definition}
\newtheorem*{Definition}{Definition}

\theoremstyle{plain}
\newtheorem{lem}[thm]{Lemma}
\newtheorem{prop}[thm]{Proposition}
\newtheorem{cor}[thm]{Corollary}

\newtheorem{remark}[thm]{Remark}
\newtheorem*{Remark}{Remark}

\theoremstyle{remark}

\begin{document}

\pagenumbering{roman}

\thispagestyle{empty}

\begin{center}
    \Large
    Ben Gurion University of the Negev\\
    The Faculty of Natural Sciences\\
    Department of Mathematics
\end{center}

\vfill\vfill
\begin{center}
    \large
    Master of Science Thesis\\
\end{center}

\vfill
\begin{center}
    \Huge\bfseries
		Looking for a Billiard Table which is not a Lattice Polygon but Satisfies Veech's Dichotomy
\end{center}

\vfill
\begin{center}
    \Large
    by
\end{center}

\vfill
\begin{center}
    \huge\bfseries
    Meital Cohen
\end{center}

\vfill\vfill\vfill
\begin{center}
    \Large
    Supervisor: Prof. Barak Weiss
\end{center}

\vfill
\begin{center}
\large
    Beer Sheva, 2010
\end{center}

\cleardoublepage

\pagestyle {headings}
\newpage
\section{Abstract}
Many problems in geometry, topology and dynamical systems on surfaces lead to the study of closed surfaces endowed with a flat metric containing several cone-type singularities, that are called \textit{flat structures} or \textit{translation surfaces}.\\

The first construction in the theory is that of a flat structure which is obtained from billiard on a rational polygon by \textit{unfolding} it to a collection of polygons with gluings. That is, from a polygon $P$, we develop a polygonal system $\widetilde{P}$, with sides identified by translation, to get a surface $M = \widetilde{P} \ / \backsim$. The idea is to associate a closed, orientable surface to the billiard table which has the same geodesics (the billiard trajectories). Rational billiard determines a flat structure; however most of the flat structures are not obtained from billiards.\\

Over the course of studying billiard dynamics, several questions were raised. One of the questions was, which surfaces satisfy the following property (which is called \textit{Veech's dichotomy}): Any direction is either \textit{completely periodic} or \textit{uniquely ergodic}.\\

In an important paper \cite{Veech89} Veech gave a sufficient condition for this dichotomy. He showed that if the stabilizer of a translation surface is a lattice in $SL_2(\R)$, then the surface satisfies Veech's dichotomy. Later, Smillie and Weiss \cite{SW} proved that this condition is not necessary. They constructed a translation surface which satisfies Veech's dichotomy but is not a lattice surface. Their construction was based on previous work of Hubert and Schmidt \cite{HS infty}, by taking a branched cover over a lattice surface, where the branch locus is a single \textit{non-periodic connection point}.\\

In this work we tried to answer the following question: Is there a flat structure obtained from a billiard table that satisfies Veech's dichotomy, but its Veech group is not a lattice? In Theorem \ref{main} we prove that in the entire list of possible candidates for such a construction, an example does not exist.

\section{Acknowledgements}

This is my M.Sc. thesis at Ben Gurion University. 
This document is not being submitted for publication. 
This research was supported by the Israel Science Foundation.\\

I would like to thank all the people who helped, supported and inspired me during my studies.\\

First and foremost I would like to thank my supervisor, Prof. Barak Weiss, for his incredible support, patience and endless encouragement. Barak was always available to me and willing to help. His great ability to explain clearly and simply, made challenges more accessible to me. His vast knowledge, skills and pleasant personality inspired me during the whole period we have worked together.\\

Furthermore, I would like to express my gratitude to my friends from BGU math department, Yaar, Daniel and Hai. They provided me with a pleasant environment, good advice, assistance and encouragement at difficult times. In addition, I want to thank the Administrative Coordinator Rutie Peled, who welcomed me with opened arms and a big smile. Rutie always was glad to help with every problem I had during 
my studies and work in the department.\\

Finally, my deepest gratitude goes to my dearest companion Dror for his abundant love and support throughout this loaded period. 
This work was simply impossible without him. 

\newpage

\tableofcontents
\newpage
\pagenumbering{arabic}

\section{Background}
General references for translation surfaces and billiards are \cite{MaTa}, \cite{HS intro}, \cite{Vor96}.

\subsection{Translation Surfaces}
\begin{Definition}
A \textit{translation surface} or \textit{flat structure} is a finite union of Euclidean polygons $\{P_1,P_2, ..., P_n\}$ with identifications such that:
\begin{enumerate}
\item The boundary of every polygon is oriented such that the polygon lies to the left.
\item For every $1\leq j \leq n $, for every oriented side $s_j$ of 	$P_j$ there exists $1\leq k \leq n $ and an oriented side $s_k$ of $P_k$ such that $s_j$ and $s_k$ are parallel, of the same length and of opposite orientation. They are glued together in the opposite orientation by a parallel translation.
\end{enumerate}
\end{Definition}

There is a finite set of points $V$, corresponding to the vertices of the polygons. The \textit{cone angle} at a point in $V$ is the sum of the angles at the corresponding points in the polygons $P_i$. For each point in $V$ the total angle around the point is $2 \pi k$ where $k \in \N$. We say that a point is \textit{singular} of \textit{multiplicity} $k$ if $k>1$, otherwise it is called \textit{regular}. We will denote by $\Sigma=\Sigma_M$ the set of the singular points.\\

If $\backsim$ denotes the equivalence relation coming from identification of sides then we define a closed oriented surface $M=\bigcup P_j / \backsim$ with a finite set of points $\Sigma$, \text{corresponding} to the singular points of $M$.\\

There is an equivalent definition of translation surfaces in terms of special atlases called \textit{translation atlases}: A translation surface is a compact orientable surface with an atlas such that, away from finitely many points called singularities, all transition functions are translations.

\begin{prop} \cite{Vor96}
Let $M$ be a translation surface, let $k_1, k_2, \ldots k_m$ be the multiplicities of its singular points. Then $2g-2=-\chi(M)=\sum_{i=1}^{m}(k_i-1)$ where $g$ is the genus of $M$ and $\chi(M)$ is the Euler characteristic of $M$.
\end{prop}

\subsection{Relation to Billiards}
\subsubsection {The Unfolding Process for Rational Billiards:}
We will describe the unfolding process, which gives the motivation to the following definition that describes the connection of billiards to translation surfaces.\\

Let $P$ be a rational polygon (all of whose angles are rational multiples of $\pi$). Given a billiard trajectory (that avoids the vertices) beginning at a side of $P$. Given a collision with a side we reflect the polygon along the side, obtaining a mirror image of the original polygon, on which the billiard now continues in its original direction, instead of reflecting off the side. Continuing this process, we obtain a straight line, passing copies of the polygon. Since $P$ is a rational polygon, there are only finitely many possible angles of incidence of our trajectory with these copies. Thus, the billiard eventually exits a copy of the polygon in a side that is parallel with the initial side. We identify these sides by translation and continue this process, considering any unpaired side that the billiards meets as the new initial side. The result is a new polygon with various opposite sides identified (see Figure \ref{unfolding}). On this flat surface, the billiard moves along straight line segments, up to translation.\\

\begin{figure}[h!]
\begin{center}
\includegraphics[scale=0.7]{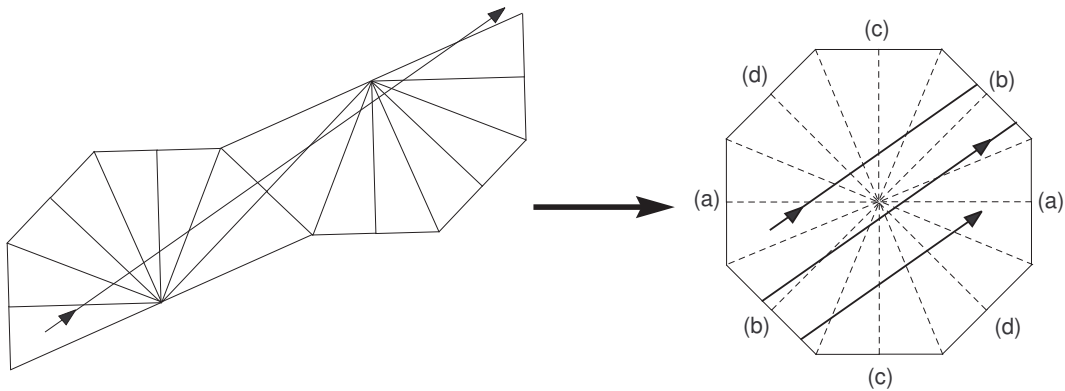}
\caption{The unfolding process of a triangle billiard table with angles $\left(\frac{\pi}{2}, \frac{\pi}{8}, \frac{3\pi}{8} \right)$
	and the flat structure obtained - the octagon with parallel sides identified.}
\label{unfolding}			
\end{center}
\end{figure}

\subsubsection {Rational billiard determines a translation surface:}
Let $A_P$ be the group of motions of the plane generated by the reflections in the sides of $P$. It follows that every copy of $P$, involved in the unfolding, is the image of $P$ under an element of the group $A_P$. The product of an even number of elements of this group preserves orientation while an odd number reverses it. To keep track of the directions of billiard trajectories in $P$ consider the group $G_P$ that consists of the linear parts of the motions from $A_P$. This subgroup of the orthogonal group is generated by reflections in the lines through the origin, which are parallel to the sides of the polygon $P$.\\

Let $P$ be a simply connected rational polygon with angles $\frac {m_i}{n_i}\pi$ where $m_i$ and $n_i$ are coprime integers, and denote $N=lcm\{n_i\}$. Then the group $G_P$ is the dihedral group $D_N$ with $2N$ elements ($N$ reflections and $N$ rotations), denote $G_P=\{g_1, g_2, \ldots g_{2N} \}$. Consider $2N$ disjoint copies of $P$ in the plane, and denote $P_k=g_k P \ ,\  k=1,2, \dots, 2N$. If $g_k$ preserves orientation, then orient $P_k$ clockwise, else orient $P_k$ counterclockwise. Now, paste their sides together pairwise: For each side $e_i^{k_1}$ of $P_{k_1}$ paste the side $e_j^{k_2}$ of $P_{k_2}$ such that $P_{k_2}$ is obtained from $P_{k_1}$ by reflection with respect to the side $e_i^{k_1}$. After these pastings are made for all the sides of all the polygons, one obtains a translation surface $M_P$. 

\begin{remark}
\label{gluing}
The point in $M$ corresponding to a vertex in $P$ is the result of gluing $2n_i$ copies of the angle $\frac{m_i}{n_i}\pi$ which sums up to an angle of $2 \pi m_i$. It follows that a vertex in $P$ defines a regular point if and only if its angle is $\frac{\pi}{n}$.
\end{remark}

\begin{prop}\label{N_even}
Let $P$ be a rational polygon. If $N$ is even then $-Id \in G_P$.
\end{prop}

\begin{proof}
$G_P=D_N$ the dihedral group with $2N$ elements. The kernel of the determinant map on $D_N$ is an index two subgroup of rotations, which is generated by the rotation with angle $\frac{2\pi}{N}$. Hence, if $N$ is even we have $-Id \in D_N$.
\end{proof}

We will call a rational angle with even denominator an \textit{even angle}.

\begin{cor}
\label{-id}
Let $P$ be a rational polygon. If one of the angles in $P$ is even or if there exist two external sides in $P$ with an even angle between them, then $-Id \in G_P$.
\end{cor}

\subsection{$SL_2^\pm(\R)$ action, Aff($M$) and Veech Groups}
Let $SL_2^\pm(\R)$ denote the group of $2 \times 2$ matrices with determinant $\pm 1$. Given any matrix $A\!\in\!SL_2^\pm(\R)$ and a translation surface $M$, we can get a new translation surface $A\!\cdot\!M$ by post composing the charts of $M$ with $A$. The transition functions of $A\!\cdot\!M$ are translations since they are the transition functions of $M$ conjugated by $A$. 

\begin{Definition}
Let $M_1$ and $M_2$ be translation surfaces. A homeomorphism $f:M_1 \to M_2$ is called an \textit{isomorphism of flat structures} if it maps the singular points of $M_1$ to the singular points of $M_2$ and is a translation in the local coordinates of $M_1$ and $M_2$.
\end{Definition}

\begin{Definition}
An \textit{affine diffeomorphism} of a translation surface $M$ is a homeomorphism $f:M\to M$ that maps singular points to singular points and is an affine map in the local coordinates of the atlas of $M$. This is equivalent to the fact that $f$ is an isomorphism of the flat structures $A\!\cdot\!M$ and $M$, where $A\!\in\!SL_2^\pm(\R)$. $A$ is called the \textit{linear part} of the automorphism $f$.
\end{Definition}

Denote by Aff($M$) the affine automorphism group of $M$. Following the construction in $\S 3.2.2$ of flat structure $M_P$ obtained from the 
billiard in polygon $P$, $G_P \subseteq \text{Aff}(M_P)$, since for each $g_i \in G_P$, $g_i M_P=M_P$. 

For example, one can easily see this action in Figure \ref{unfolding}. In this figure, we have the octagon, the surface $M_P$ obtained from the billiard in the triangle $P=\left(\frac{\pi}{2}, \frac{\pi}{8}, \frac{3\pi}{8} \right)$.$M_P$ made up of all the translations of $P$ by $g\in G_P$, i.e. $M_P=\cup_{i=0}^{15}g_iP$. Hence, $\forall g\in G_P$ we have: $gM_P=g\cup_{i=0}^{15}g_iP=\cup_{i=0}^{15}\tilde{g_i}P=M_P$.

\begin{Definition}
The \textit{Veech group} $\Gamma_M$ of a flat structure $M$ is the set of elements $A\!\in\!SL_2^\pm(\R)$ such that the flat structures $A\!\cdot\!M$ and $M$ are isomorphic. 
\end{Definition}

The following property of Veech groups is well known -

\begin{prop}
$\Gamma_M$ is a discrete and non-cocompact subgroup of $SL_2^\pm(\R)$.
\end{prop}

We call a translation surface a \textit{lattice surface} if $\Gamma_M$ is a lattice in $SL_2^\pm(\R)$. If $P$ is a polygon such that $M_P$ is a lattice surface, we say that $P$ has the lattice property.\\

In view of the above, the Veech group $\Gamma_M$ is the image of Aff($M$) under the derivative map \text{$D:\text{Aff}(M) \to \Gamma_M$}. The derivative map has a finite kernel \cite{Veech89}. For a translation surface $M$, denote ker$(D)$ by Trans$(M)$. Therefore we have an exact sequence: \[	1 \longrightarrow \text{Trans}(M) \xrightarrow{\;\;\, i \;\; } 
\text{Aff}(M) \xrightarrow{\;\: D \;\; } \Gamma_M \longrightarrow 1 \]
There are translation surfaces with \text{Trans$(M)\neq \{Id\}$}, for those surfaces we cannot identify Aff$(M)$ with $\Gamma_M$. 

\subsection{Veech's Dichotomy}
Fix a translation surface $M$. A starting point in $M$ and an angle $\theta$ determine a trajectory. A trajectory which begins and ends at a singular point is called a \textit{saddle connection}. A trajectory which does not hit singular points is called \textit{infinite}. A direction in which all infinite trajectories are dense is said to be \textit{minimal}. A direction in which all infinite trajectories are uniformly distributed is said to be \textit{uniquely ergodic}. A direction in which all orbits are periodic or saddle connections is said to be \textit{completely periodic}. Every periodic trajectory is contained in a maximal family of parallel periodic trajectories of the same period. If the surface is not a torus, then this family fills out a \textit{cylinder} bounded by saddle connections.
\begin{figure}[h!]
\begin{center}
\includegraphics[scale=0.7]{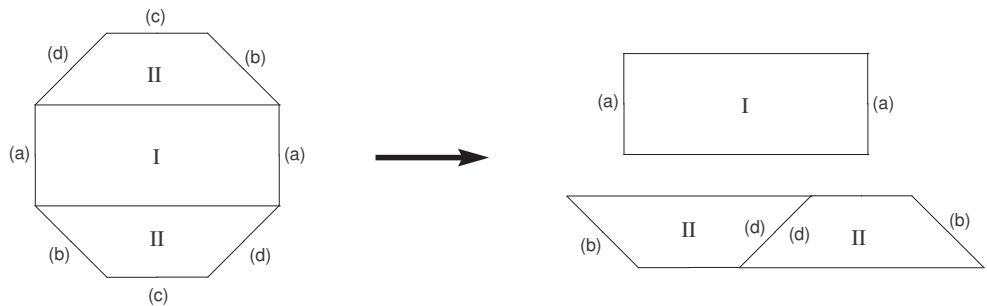}
\caption{Cylinder decomposition of the regular octagon in the horizontal direction.}
\label{cylinder}			
\end{center}
\end{figure}

\begin{thm} \cite{ZK}, \cite{BKM} 
\label{ZK}
If there are no saddle connections in direction $\theta$, then direction $\theta$ is minimal.
\end{thm}

{\bf Veech's dichotomy:} A translation surface satisfies \textit{Veech's dichotomy} if each direction is either completely periodic or uniquely ergodic.

\begin{Thm} (Veech's theorem, \cite{Veech89}) 
Suppose that $M$ is a lattice surface. Then $M$ satisfies Veech's dichotomy. 
\end{Thm}

In \cite{CHM} there are definitions for different dichotomies of translation surfaces: A translation surface is said to satisfy \textit{topological dichotomy} if for every direction - if there is a saddle connection in direction $\theta$, then there is a cylinder decomposition of the surface in that direction. A translation surface is said to satisfy \textit{strict ergodicity} if every minimal direction is uniquely ergodic. Lattice surfaces satisfy both topological dichotomy and strict ergodicity.

\subsection{Covers of Flat Structures}

\subsubsection{Riemann-Hurwitz formula}
The Riemann-Hurwitz formula describes the relation between the Euler characteristics of two surfaces when one is a ramified covering of the other.

\begin{Definition}
The map $\pi:\widetilde{M} \to M$ is said to be \textit{ramified} at a point $p \in \widetilde{M}$, if there exists a small neighborhood $U$ of $p$, such that $\pi(p)$ has exactly one pre-image in $U$, but the image of any other point in $U$ has exactly $n>1$ pre-images in $U$. The number $n$ is called the \textit{ramification index} at $p$, also denoted by $e_p$. The map $\pi$ is said to be \textit{branched} over a point $z \in M$ if $\pi$ is ramified at a point of $\pi^{-1}(z)$. 
\end{Definition}

The \textit{degree} of the map $\pi:\widetilde{M} \to M$, denoted by $d$, is the number of pre-images of non-singular points in $M$ (independent of the point). For an orientable surface $M$, the Euler characteristic $\chi(M)$ is the number $2-2g$, where $g$ is the genus of $M$. In the case of an unramified covering map of surfaces $\pi:\widetilde{M} \to M$ that is surjective and of degree $d$, we get the formula: \[\chi(\widetilde{M})=d\chi(M)\] 
The Riemann-Hurwitz formula adds a correction for ramified covers. In calculating the Euler characteristic of $\widetilde{M}$ we notice the loss of $e_p-1$ copies of $p$ above $\pi(p)$ (for more explanations see \cite{Wiki}). Therefore the corrected formula for ramified covering is: \[\chi(\widetilde{M})=d\chi(M)-\sum_{p \in \widetilde{M}}(e_p-1)\]

\subsubsection{Translation covering}
Let $M'=M\setminus \Sigma_M$. We say that a map $\pi:M \to N$ gives a \textit{translation covering} of $N$ by $M$, if the restriction $\pi:M' \to N'$ is such that $\psi \circ \pi \circ \phi^{-1}$ are translations, where $\psi$ and $\phi$ are the local coordinates of the atlases of $N$ and $M$. Note that if a translation covering is ramified, it is ramified over singularities or marked points only.\\

The following construction gives examples of covers of flat structures associated with billiards in polygons: Let $P$ and $Q$ be rational polygons such that $P$ tiles $Q$ by reflections. This means that $Q$ is partitioned into a finite number of isometric copies of $P$. Each two are either disjoint, or have a common vertex or a common side, and if two of these polygons have a side in common, then they are symmetric with respect to this side. The construction in $\S 3.2$ associates with $P$ and $Q$ translation surfaces $M_P$ and $M_Q$.

\begin{prop} \cite{Vor96} 
The flat structure $M_Q$ covers $M_P$, possibly after adding to $M_Q$ or removing from $M_P$ a certain number of removable singular points. 
\end{prop}

\begin{remark}
\label{observation}
\textbf{Some observations:}
\begin{enumerate}
	\item Let $P$ and $Q$ be polygons such that $P$ tiles $Q$ by reflections, then the branch points of the covering map 
		$\pi:M_Q \to M_P$ can arise only from the vertices in $P$.
	\item One can see that the degree of the cover is $d=\frac{n}{m}$, where $n$ is the number of copies of $P$ in $Q$ and $m$ is the 
		index of the subgroup $G_Q$ in $G_P$.
	\item $G_Q$ is a subgroup of $G_P$ and by $\S 3.2.2$ $G_Q$ is also a dihedral group. Hence $G_Q$ must be one of the dihedral groups 
		$D_{N_i}$, where $N_i$ is a divisor of $N$.
\end{enumerate}
\end{remark}

\begin{prop}
\label{branching}
Suppose $P$ and $Q$ are polygons such that $P$ tiles $Q$ by reflections. 
\begin{enumerate}
\item Let $z_0 \in M_P$ be a point which corresponds to a vertex with angle $\frac{m_0\pi}{n_0}$ in $P$ with $(m_0,n_0)=1$.
\item $y_i \in \pi^{-1}(z_0) \subset M_Q$, $i=1, \ldots l$ the pre-images of $z_0$ which corresponds to vertices with angles $\frac{k_i\cdot m_0 \pi}{n_0}$ in $Q$. 
\end{enumerate}
Then the cover $\pi: M_Q \to M_P$ will branch over $z_0$ if and only if there exists $i_0$ such that $k_{i_0} \nmid n_0$.
\end{prop}

\begin{proof}
The covering map $\pi: M_Q \to M_P$ will branch over a point $z_0 \in M_P$ if and only if the total angle around $z_0 \in M_P$ differs from the total angle around one of the pre-images $\pi^{-1}(z_0) \in M_Q$. Following Remark \ref{gluing}, the total angle of $z_0$ is $2m_0\pi$, and the total angle of $\pi^{-1}(z_0)$ is $\frac{2k_im_0}{gcd(k_i,n_0)}\pi$. Therefore, they are equal if and only if for all $i$, $gcd(k_i,n_0)=k_i \Leftrightarrow k_i \mid n_0$.
\end{proof}

\begin{Definition}
A \textit{periodic point} on a translation surface $M$ is a point which has a finite orbit under the group $\text{Aff}(M)$.
\end{Definition}

In particular, since every $\varphi \in \text{Aff}(M)$ maps singular points to singular points, and $\Sigma$ is finite, any singular point of $M$ is periodic.\\

According \cite{GJ96}, recall that two groups 
$\Gamma_1, \Gamma_2 \subset SL_2(\R)$ are \textit{commensurable}, if there exists $g \in SL_2(\R)$ such that the group 
$\Gamma_1 \cap g\Gamma_2 g^{-1}$ has finite index in both $\Gamma_1$ and $g\Gamma_2 g^{-1}$.

\begin{prop} \label{commensurable} \cite{Vor96} \cite{GJ96}
Let $M_P$ and $M_Q$ be as above. Then $\Gamma_{M_Q}$ is commensurable to the stabilizer in $\Gamma_{M_P}$ of the branch locus of $\pi$. 
Thus, if $\Gamma_{M_P}$ is a lattice in $SL_2(\R)$, $\Gamma_{M_Q}$ is also a lattice in $SL_2(\R)$ if and only if the branching is 
over periodic points.
\end{prop}

\begin{remark}
\label{non-periodic}
The last proposition along with Lemma 4 in \cite{HS covering} give us a test for non-periodicity of a point in a lattice surface: A point in a lattice surface $M$ which irrationally splits the height of a cylinder of $M$ is a non-periodic point.
\end{remark}

\subsection{The Converse to Veech's Theorem Does Not Hold}
In genus 2, McMullen showed that every surface which is not a lattice, does not satisfy Veech's dichotomy \cite{Mc decagon}. In other words, if a surface is of genus 2, then it satisfies Veech's dichotomy if and only if its Veech group is a lattice.

\begin{Thm} \cite{SW} 
There is a flat structure which satisfies Veech's dichotomy but its Veech group is not a lattice.
\end{Thm} 

This result relied on previous work of Hubert and Schmidt. In their work they defined the notion of a non-periodic connection point on a flat structure, and proved that for some surfaces there exist infinitely many non-periodic connection points.

\begin{Definition}
A nonsingular point $p$ is a \textit{connection point} of a translation surface $M$, if every geodesic emanating from a singularity passing through $p$, is a saddle connection.
\end{Definition}

\begin{Thm} \cite{HS infty} 
There are lattice surfaces of genus 2 and 3 containing infinitely many non-periodic connection points. If $M$ is such a translation surface and $\widetilde{M}$ is a translation surface obtained by forming a cover of $M$ branched only at non-periodic connection points, then $\widetilde{M}$ is not a lattice surface, yet has the property that any direction is either completely periodic or minimal.
\end{Thm}

Weiss and Smillie show that provided the branching takes place over a single point (not necessarily a connection point), the minimal directions are uniquely ergodic. Thus, these surfaces satisfy Veech's dichotomy, though their Veech groups are not lattices.\\

Following the definitions in \cite{CHM}, Smillie and Weiss show that a surface satisfying both topological dichotomy and strict ergodicity need not be a lattice surface. The construction of Smillie and Weiss proves strict ergodicity for every cover over a lattice surface branched over one point. If the branch point is not a connection point, then the cover does not satisfy the topological dichotomy. This means that there are also examples that satisfy strict ergodicity and not topological dichotomy.\\

By the Riemann-Hurwitz formula, one can show the smallest genus, for which the arguments of \cite{SW} work, is 5. If $d=2$ and there is a single branch point, then the cover is ramified at one point with $e_p=2$, and the corresponding Riemann-Hurwitz formula is: $\chi(\widetilde{M})=2\chi(M)-1$. Since $\chi(\widetilde{M})$ is even, this cannot occur. So take $d\geq 3$, and denote by $g'$ the genus of $\widetilde{M}$. We get: \[\chi(\widetilde{M})=d\chi(M)-\sum_{p \in \widetilde{M}}(e_p-1)\] Since $\sum_{p \in \widetilde{M}}(e_p-1)>0 $ we get: \[d(2-2g)-(2-2g')=\sum_{p \in \widetilde{M}}(e_p-1)> 0 \qquad \Longrightarrow \qquad g'>1+d(g-1)\] Hence, the smallest genus $g'$ which satisfies the inequality is 5, which is obtained with $g=2$ and $d=3$.

\begin{Remark}
We ignore the possibility of $g=1$, since in that case, the cover is a square-tiled surface, and hence a lattice (see \cite{GJ00}).
\end{Remark}

\newpage

\section{The Main Result}

\textbf{\large{The Question I Explored}}
\textit {\large: Is there a flat structure obtained from a billiard table that satisfies Veech's dichotomy, yet its Veech group is not a lattice?}\\

In order to find an example, we tried to follow the construction in \cite{SW}. Therefore, we will look for polygons $P$ and $Q$ such that:
\begin{enumerate}
\item $P$ has the lattice property.
\item $P$ tiles $Q$ by reflections.
\item The branched covering map $\pi :M_Q \to M_P$ is branched over a single non-periodic connection point.
\end{enumerate}

So far we can barely indicate connection points. 
However, the following proposition shows that Veech's 
dichotomy holds, even if we omit the requirement that the 
branch point will be a connection point. Following this 
result, we can expand our search to find a cover which is 
branched only over one non-periodic point.

\begin{definition} 
We say that a cover $\pi:M_Q \to M_P$ is an 
\textit{appropriate cover} if:
\begin{enumerate}
\item $P$ be a lattice polygon.
\item $P$ tiles $Q$ by reflections.
\item The branched covering map $\pi:M_Q \to M_P$ 
		is branched over a single non-periodic point.
\end{enumerate}
\end{definition} 

\begin{prop}
\label{dichotomy}
Let $P$ and $Q$ be as above such that the cover $\pi :M_Q \to M_P$ 
is branched over a single point.
If the branching is over a single point it satisfies Veech dichotomy.
Moreover, if the point is non-periodic, it is not a lattice surface. 
\end{prop}

\begin{proof}
We need to show that any direction in $M_Q$ is either completely periodic or uniquely ergodic. According to \cite{SW}, provided the branching is over a single point, the minimal directions are uniquely ergodic. Therefore, we need to prove that the remaining directions are either completely periodic or uniquely ergodic.

According to Theorem \ref{ZK}, if a direction in $M_Q$ is not minimal, then there is a saddle connection in this direction. There are three types of saddle connections on $M_Q$:
\begin{enumerate}
\item Those that project to a saddle connection on $M_P$.
\item Those that project to a geodesic segment connecting $p$ to itself.
\item Those that project to a geodesic segment connecting $p$ to a singularity.
\end{enumerate}

Notice that any completely periodic direction in $M_P$ is a completely periodic direction in $M_Q$. Since $M_P$ is a lattice surface, the first category of directions is completely periodic directions. The second category is obviously periodic directions in $M_P$. For the third category we have 2 possibilities: If it projects to a saddle connection on $M_P$, then it is also a periodic direction. Else, the direction must be minimal in $M_P$, hence uniquely ergodic in $M_Q$.

If the covering map $\pi:M_Q \to M_P$ is branched over a non-periodic point, then Proposition \ref{commensurable} 
implies that the polygon $Q$ does not have the lattice property.
\end{proof}

\begin{Remark}
In case $p$ is a connection point, the Veech dichotomy on $M_Q$ permits a strong formulation: $M_Q$ satisfies strict ergodicity and topological dichotomy. In particular, the completely periodic directions are precisely the saddle connection directions.
\end{Remark}

\begin{cor}
An appropriate cover gives rise to a billiard polygon satisfying Veech's dichotomy without the lattice property.
\end{cor}

\subsection{List of all known lattice surfaces coming from polygons}

For the moment there is no classification of lattice polygons.
The following list contains the known lattice polygons:

\begin{enumerate}
\item Regular polygons \cite{Veech89}.
\item Right triangles with angles $\left(\frac{\pi}{2}, \frac{\pi}{n}, \frac{(n-2)\pi}{2n} \right)$ for $n\geq 4$ \quad \cite{Veech89}, \cite{Vor96}, \cite{KeSm}.
\item Acute isoceles triangles with angles $\left(\frac{(n-1)\pi}{2n}, \frac{(n-1)\pi}{2n}, \frac{\pi}{n}\right)$ for $n\geq 3$ \ \cite{Veech89}, \cite{Vor96}, \cite{GJ00}, \cite{KeSm}.
\item Obtuse isosceles triangles with angles $\left(\frac{\pi}{n}, \frac{\pi}{n}, \frac{(n-2)\pi}{n}\right)$ for $n\geq 5$ \cite{Veech89}.
\item Acute scalene triangles 
		$\left(\frac{\pi}{4}, \frac{\pi}{3}, \frac{5\pi}{12}\right)$,
		$\left(\frac{\pi}{5}, \frac{\pi}{3}, \frac{7\pi}{15}\right)$ and
		$\left(\frac{2\pi}{9}, \frac{\pi}{3}, \frac{4\pi}{9}\right)$
		\quad \cite{Veech89}, \cite{Vor96}, \cite{KeSm} respectively.
\item Obtuse triangles with angles $\left(\frac{\pi}{2n}, \frac{\pi}{n}, \frac{(2n-3)\pi}{2n}\right)$ for $n\geq 4$ \cite{Vor96}, \cite{Wrd98}.
\item Obtuse triangle with angles $\left(\frac{\pi}{12}, \frac{\pi}{3}, \frac{7\pi}{12}\right)$ \cite{Ho}.
\item L-shaped polygons (See \cite{Mc spin} for a description).
\item Bouw and M\"{o}ller examples: (See \cite{BM} for a description)
		\begin{itemize}
			\item 4-gon with angles $\left(\frac{\pi}{n},\frac{\pi}{n},\frac{\pi}{2n},\frac{(4n-5)\pi}{2n}\right)$ for $n\geq 7$ and odd.
			\item 4-gon with angles $\left(\frac{\pi}{2},\frac{\pi}{n},\frac{\pi}{n},\frac{(3n-4)\pi}{2n}\right)$ for $n\geq 5$ and odd.
		\end{itemize}
\item Square-tiled polygons \cite{GJ96}.
\end{enumerate}

\begin{remark}
\label{lattice_cover}
According to Proposition \ref{commensurable}, if $\overline{P}$ is a polygon which is tiled by one of the polygons $P$ in the list, such 
that the covering map $\pi: M_{\overline{P}} \to M_{P}$ is branched only over periodic points, then $\overline{P}$ has the lattice property.
\end{remark} 

\begin{thm}
\label{main}
There is no appropriate cover $\pi: M_Q \to M_P$ 
with P in the list.
\end{thm}

In order to prove this theorem, we will follow the list in $\S 4.1$ and show in each case that there is no appropriate cover. For each polygon $P_i$ in the list we will check two kinds of appropriate covers. These two cases will be referred to as appropriate covers of the first and second class respectively. 

\begin{enumerate}
\item $\pi:M_Q \to M_{P_i}$ where $P_i$ is one of the polygons in the list above.
\item $\pi:M_Q \to M_{\overline{P}}$ where $M_{\overline{P}}$ covers $M_{P_i}$ as in Remark \ref{lattice_cover}.
\end{enumerate}
Note that if all the points in the surface $M_P$ corresponding to the vertices in $P$ are periodic, we could not find an appropriate cover of neither kind.

\newpage
\section{Preliminary Lemmas}

\begin{lem}
\label{fp_of_G_Q}
Let $M_P$ and $M_Q$ be translation surfaces as in the construction in $\S 3.5.2$, such that $\pi :M_Q \to M_P$ is a branched covering map, where the branch locus is a single point $z_0 \in M_P$. Then $z_0$ is a fixed point of $G_Q$.
\end{lem}
\begin{proof}
As we mentioned in $\S 3.3$, $G_Q \subseteq \text{Aff}(M_Q)$. Since $G_Q < G_P$ we have $G_Q \subseteq \text{Aff}(M_P)$. Then, for all $g \in G_Q$, we have the following diagram:
\begin{equation*}
\begin{matrix}
\ M_Q & \xrightarrow{\quad g \quad} & M_Q \\
\pi\Big\downarrow & \ & \Big\downarrow\pi \\
\ M_P & \xrightarrow[\quad g \quad]{} & M_P
\end{matrix}
\end{equation*}
Hence, the set of branch points in $M_P$ is $G_Q$-invariant. Therefore, if $z_0\in M_P$ is a single branch point, it is a fixed point of $G_Q$.
\end{proof}

\begin{remark}
\label{remark_fp_of_G_Q}
In fact, if we denote $H=\{\ h\in G_P \ | \ hM_Q=M_Q \}$, then the branch point must be a fixed point of the group $\langle G_Q, H\rangle$. 
Moreover, if we want to branch over a single point $z_0 \in M_P$, which is not a fixed point of $h \in G_P$, then $h \notin G_Q$.
\end{remark}

\begin{lem}
\label{finite_fixed_points}
Let $\varphi \in \text{Aff}(M)$ such that $D\varphi \in SL_2^\pm(\R)$ is elliptic or hyperbolic. 
Then $\varphi$ has only a finite number of fixed points.
\end{lem}

\begin{proof}
Suppose by contradiction that $\varphi$ has an infinite number of 
fixed points in $M$. Since $M$ is compact, there exists a Cauchy 
sequence of fixed points: 
$\{x_n\}_{n=1}^{\infty}\subset M\setminus \Sigma_M$, with 
$\varphi(x_n)=x_n$ for all $n$. In particular, for all $\delta>0$, 
there exist $n, m \in \N$ such that $d(x_n,x_m)<\delta$.

M is compact, therefore we can cover $M \setminus \Sigma_M$ with 
finitely may sets $\{U_i\}_{i=1}^k$ such that, for any $x, y \in U_i$, 
there exists a geodesic from x to y. Let $r$ be the distance such that if 
$x,y \in M$ with $d(x,y)<r$, then there exists one geodesic on $M$ 
connecting between $x$ and $y$ of length less than $r$. The 
compactness of $M$ implies that $r$ can be chosen uniformly. 
Let $K$ be a Lipschitz constant of $\varphi$, and 
$A=D\varphi \in SL_2^{\pm}(\R)$. 

Define $\delta_0= \frac{r}{K}$. There exist $x_{n_0}, x_{m_0} \in U_{i_0}$ 
such that $d(x_{n_0},x_{m_0})<\delta_0$. 
Let $\gamma: [0,1] \to M$ be the geodesic such that 
$\gamma(0)=x_{n_0}$ and $\gamma(1)=x_{m_0}$.
There exists a map ($U_{i_0}, \psi_{i_0}$) such that 
$\gamma \subset U_{i_0}$ and 
$\psi_{i_0}(\gamma)=v \in \R^2$. 
\pagebreak

Consider $\varphi(\gamma)$: 
\begin{itemize}
\item [-] $\varphi(\gamma(0))=\varphi(x_{n_0})=x_{n_0}=\gamma(0)$
\item [-] $\varphi(\gamma(1))=\varphi(x_{m_0})=x_{m_0}=\gamma(1)$
\item [-] $\forall t_1, t_2 \in [0,1]$: \ 
		$d\left(\varphi(\gamma(t_1), \varphi(\gamma(t_2)\right) \leq 
			K \cdot d\left(\gamma(t_1), \gamma(t_2)\right) \leq
			K \cdot d(x_{n_0}, x_{m_0}) < K \cdot \delta_0 < r $, 
\end{itemize}
This implies that $\varphi(\gamma)=\gamma$. Since $\varphi \in \text{Aff}(M)$
with $A=D\varphi$ we get $\psi_\alpha(\varphi(\gamma))=A \cdot \psi_\alpha(\gamma)$.
Hence, we have: 
\[v=\psi_\alpha(\gamma)=\psi_\alpha(\varphi(\gamma))=A \cdot \psi_\alpha(\gamma)=A\cdot v\]
Consequently, $v$ is eigenvector of $A$. This is possible only if $A$ is 
parabolic. \newline A contradiction.
\end{proof}

\begin{lem}
\label{involution}
Let $M$ be a translation surface and $z_0 \in M$, and let $\sigma \in \text{Aff}(M)$ such that $\sigma(z_0)=z_0$
and $D(\sigma)=-Id\in \Gamma_M$. Then $z_0$ is periodic.
\end{lem}
\begin{proof}
Denote $A=\{z \in M \ | \ \exists \gamma \in \text{Trans}(M) \ \text{such that}\ \gamma\sigma z=z\}$. Trans(M) is finite and $D(\gamma\sigma)=D(\sigma)=-Id$, hence by Lemma \ref{finite_fixed_points} $A$ is finite. Let $\psi \in \text{Aff}(M)$, we will show that $\psi(z_0) \in A$, and since $A$ is finite, the claim is obtained. Since $D(\sigma)=-Id$, we get $D(\psi\sigma)=D(\psi)D(\sigma)=D(\sigma)D(\psi)=D(\sigma\psi)$. Therefore there exists $\gamma_0 \in ker(D)$ such that $\psi\sigma=\gamma_0\sigma\psi$. Now, $\psi(z_0)=\psi(\sigma(z_0))=\gamma_0\sigma\psi(z_0)$ which implies that $\psi(z_0) \in A$.
\end{proof}

\begin{Remark}
Lemma \ref{involution} improves Theorem 10 of \cite{GHS}.
Similar arguments show that the set of all the fixed points 
of all maps in $\text{Aff}(M)$ with derivative $-Id$ is 
$\text{Aff}(M)$-invariant. This set contains the set of 
Weierstrass points. Therefore, Theorem 10 of \cite{GHS}
is obtained without the assumption that $\text{Aff}(M)$ 
is generated by elliptic elements.
\end{Remark}

The next corollary immediately follows from the previous lemmas.

\begin{cor}
\label{-id_in_G_Q}
Let $M_P$ and $M_Q$ be translation surfaces as in the construction in $\S 3.5.2$, such that $\pi :M_Q \to M_P$ is a branched covering map, where the branch locus is a single point $z_0 \in M_P$. If $-Id \in G_Q$, then $z_0$ is a periodic point.
\end{cor}

The following corollary is obtained from Corollary \ref{-id} and the last Corollary \ref{-id_in_G_Q}.

\begin{cor}
\label{not_appropriate}
Let $P$ and $Q$ be polygons such that $P$ tiles $Q$ by reflections. If $Q$ has an even angle or if there exist two external sides with even angle between them, then $\pi: M_Q \to M_P$ is not an appropriate cover.
\end{cor}

\begin{lem}
\label{non-periodic_points}
The following points are non-periodic points:
\begin{enumerate}
\item The points corresponding to the angle $\frac{\pi}{n}$ in the surface obtained from the triangle with angles 
		$\left(\frac{\pi}{2}, \frac{\pi}{n}, \frac{(n-2)\pi}{2n} \right)$, $n\geq 5$ and odd.
\item The points corresponding to the angle $\frac{\pi}{3}$ in the surface obtained from the triangle with angles 
		$\left(\frac{\pi}{4}, \frac{\pi}{3}, \frac{5\pi}{12}\right)$.
\item The points corresponding to the angle $\frac{\pi}{3}$ in the surface obtained from the triangle with angles 
		$\left(\frac{2\pi}{9}, \frac{\pi}{3}, \frac{4\pi}{9}\right)$.
\item The points corresponding to the angles $\frac{\pi}{5}$ and $\frac{\pi}{3}$ in the surface obtained from the 
		triangle with angles $\left(\frac{\pi}{5}, \frac{\pi}{3}, \frac{7\pi}{15}\right)$.
\item The points corresponding to the angle $\frac{\pi}{n}$ in the surface obtained from the triangle with angles 
		$\left(\frac{\pi}{2n}, \frac{\pi}{n}, \frac{(2n-3)\pi}{2n}\right)$, $n\geq 5$ and odd.
		\end{enumerate}
\end{lem}

\begin{figure}
	\includegraphics[scale=0.73]{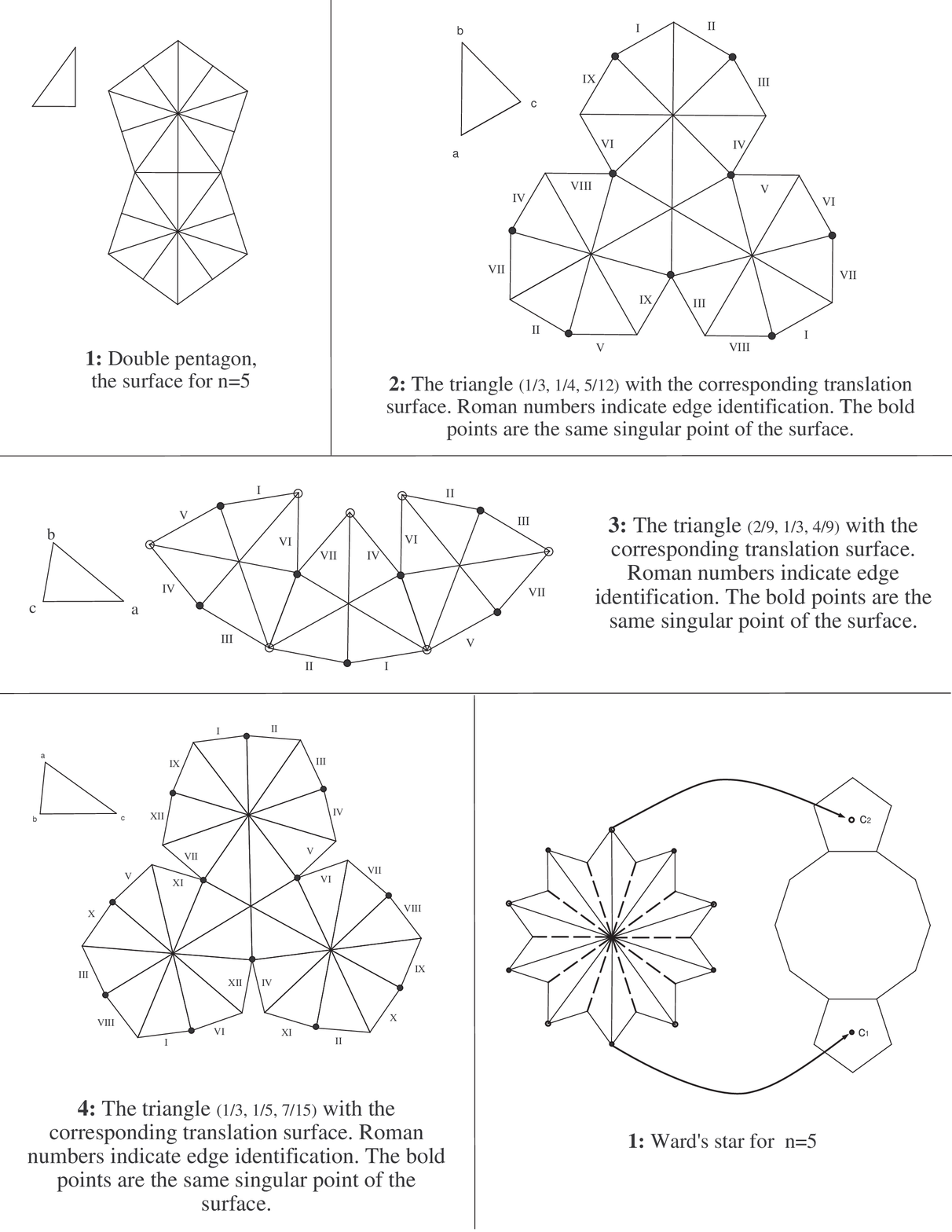}
	\caption{The surfaces in Lemma \ref{non-periodic_points}}
  	\label{surfaces_lemma}
\end{figure}

\begin{proof}
According to Remark \ref{non-periodic}, we will show that these points irrationally split the height of a cylinder of the surface. If there are two points in $M_P$ corresponding to the same angle in $P$, then both are periodic or both are non-periodic (since $G_P \subset \text{Aff}(M_P)$ swaps these points). Therefore, it is sufficient to verify this condition for only one of the points.
\begin{enumerate}
\item The surface $M_P$ obtained from the triangle with angles $\left(\frac{\pi}{2}, \frac{\pi}{n}, \frac{(n-2)\pi}{2n} \right)$ $n\geq 5$ and 
	odd,is the double regular n-gon with parallel sides identified (see Figure \ref{surfaces_lemma}, surface 1). The points corresponding to 
	the angle $\frac{\pi}{n}$ are non-periodic. This result is not new (see\cite{HS covering}, Proposition 3).
\item Denote the vertices of the triangle $P$ with angles $\frac{\pi}{3}$, $\frac{\pi}{4}$ and $\frac{5\pi}{12}$ by $a$, $b$ and $c$ 
	respectively. $M_P$ is a surface of genus 3 with one singular point which corresponds to the vertex $c$ in $P$ (see Figure 
	\ref{surfaces_lemma}, surface 2).

	In order to calculate the height ratio we will look at the cylinder decomposition of the surface in the horizontal direction, 
	and normalize one of the triangle sides to unity, as described in Figure \ref{kesm12_periodic}.
		\begin{figure}[h!]
		\begin{center}
		\includegraphics[scale=0.52]{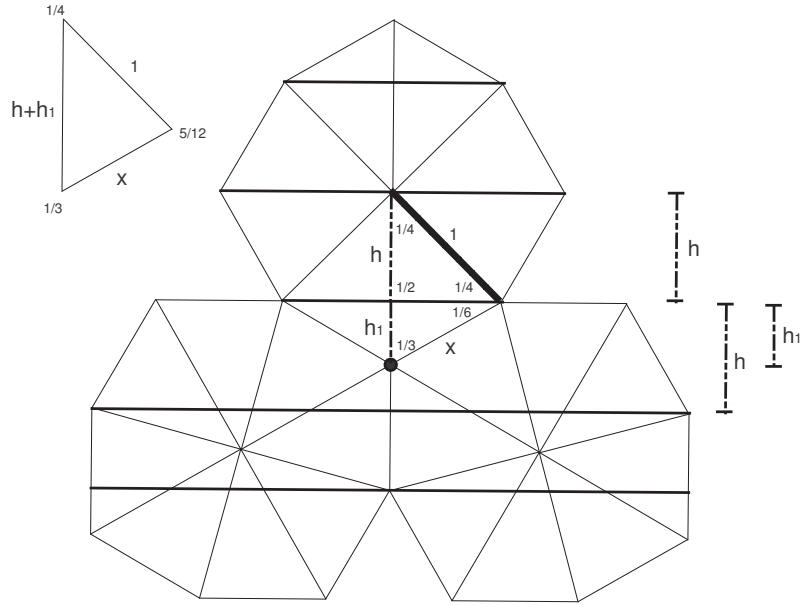}
		\caption{Cylinder decomposition in the horizontal direction of 
			the surface $M_P$ obtained from $P=\left(\frac{\pi}{4}, \frac{\pi}{3}, \frac{5\pi}{12}\right)$ -
			calculations for the point corresponding to $\frac{\pi}{3}$.}
		\label{kesm12_periodic}
		\end{center}
		\end{figure}

		The following calculation were made:
		\[h=\sin(\frac{\pi}{4})=\frac{1}{\sqrt{2}} \quad ; \quad 
		\frac{x}{\sin(\frac{\pi}{4})}=\frac{1}{\sin(\frac{\pi}{3})} \quad
			\Longrightarrow \quad x=\frac{\sqrt{2}}{\sqrt{3}}\]
		\[\frac{h_1}{\sin(\frac{\pi}{6})}=x \quad \Longrightarrow	\quad
			h_1=x\cdot \sin(\frac{\pi}{6})=\frac{\sqrt{2}}{\sqrt{3}}\cdot\frac{1}{2}\]
		Consequently, we get the irrational ratio: $\frac{h_1}{h}=\frac{1}{\sqrt{3}}\notin \Q$. 
		By Remark \ref{non-periodic}, any point corresponding to the angle $\frac{\pi}{3}$ in $P$, is non-periodic.
		\begin{Remark}
		Additional examination showed that this point is a connection point.

\end{Remark}

\item Denote the vertices of the triangle $P$ with angles $\frac{2\pi}{9}$, $\frac{\pi}{3}$ and $\frac{4\pi}{9}$ by $a$, $b$ and $c$ 
	respectively. Here $M_P$ is a surface of genus 3 with 2 singular points, corresponding to the vertices $a$ and $c$ in $P$ (see Figure 
	\ref{surfaces_lemma}, surface 3).

	In order to calculate the height ratio we will look at the cylinder decomposition of the surface in the horizontal direction, 
	and normalize one of the triangle sides to unity, as described in Figure \ref{kesm9_periodic}.
	\begin{figure}[h!]
	\begin{center}
	\includegraphics[scale=0.7]{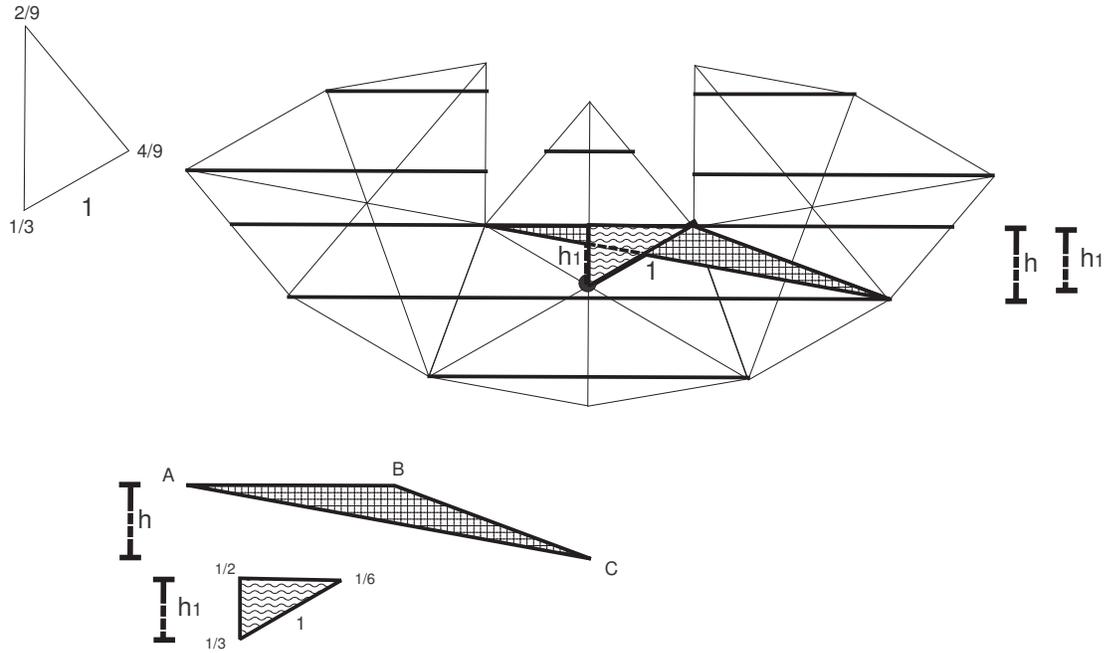}
	\caption{Cylinder decomposition in the horizontal direction of 
				the surface $M_P$ obtained from $P=\left(\frac{2\pi}{9}, \frac{\pi}{3}, \frac{4\pi}{9}\right)$ - 
				calculations for the point corresponding to $\frac{\pi}{3}$.}
	\label{kesm9_periodic}			
	\end{center}
	\end{figure}

	The following calculations were made:
	\[AB=2\cdot \sin(\frac{\pi}{3}) = \sqrt{3} \quad ; \quad
	\measuredangle{B}=\frac{8\pi}{9} \quad ; \quad 
	h_1=\cos(\frac{\pi}{3})=\frac{1}{2} \in \Q\]
	Therefore, according to Remark \ref{non-periodic}, since $h_1 \in \Q$, it remains to show that $h \notin \Q$.
	We calculate $h$ by comparing two different formulas for the area of triangle $\vartriangle{ABC}$ as follows:
	\[\frac{1}{2}\cdot h \cdot AB = \frac{1}{2} \cdot AB\cdot AB \cdot \sin(B) 
	\quad \Longrightarrow \quad h=AB \cdot \sin(B)=\sqrt{3} \cdot \sin(\frac{\pi}{9})\]
	We will show $h$ is a root of the polynomial $8x^3-18x+9$, whose roots are irrational. Using the trigonometric identity:
	\[\sin^3x=\frac{3\sin x-\sin 3x}{4}\] 
	We get: \[\sin^3(\frac{\pi}{9})=\frac{3\cdot\sin(\frac{\pi}{9})-\sin(\frac{\pi}{3})}{4}\]
	Hence, by substituting in the polynomial above we get:
	\[8 \cdot (\sqrt{3})^3 \cdot \sin^3\left(\frac{\pi}{9}\right) - 18 \cdot \sqrt{3} \cdot \sin\left(\frac{\pi}{9}\right) + 9 =
	8\cdot 3 \cdot \sqrt{3} \cdot \frac{1}{4} \cdot \left(3\cdot\sin\left(\frac{\pi}{9}\right)-\sin\left(\frac{\pi}{3}\right)\right)-18 \cdot \sqrt{3} \cdot \sin\left(\frac{\pi}{9}\right) +9 = \]
\[=18 \cdot \sqrt{3} \cdot \sin\left(\frac{\pi}{9}\right)-6\cdot \sqrt{3} \cdot \frac{\sqrt{3}}{2} -18 \cdot \sqrt{3} \cdot \sin\left(\frac{\pi}{9}\right)+9=0\]
	By the \textit{rational root test}, if $\frac{p}{q} \in \Q$ is a root of the polynomial above, than $p \mid 9$ and $q \mid 8$, 
	i.e. $\frac{p}{q} \in \{1, 3, \frac{1}{2}, \frac{1}{4}, \frac{3}{2}, \frac{3}{4} \}$.
	One can check that none of these rational numbers is a root of this polynomial, hence $h\notin \Q$.

\item Denote the vertices of the triangle $P$ with angles $\frac{\pi}{3}$, $\frac{7\pi}{15}$ and $\frac{\pi}{5}$ by $a$, $b$ and $c$ 
	respectively. $M_P$ is a surface of genus 4, with one singular point corresponding to the vertex $b$ in $P$ (see Figure 
	\ref{surfaces_lemma}, surface 4).

	\textbf{First we will examine the points corresponding to the vertex $c$:}

	In order to calculate the height ratio we will look at the cylinder decomposition of the surface in the horizontal direction, 
	and normalize one of the triangle sides to unity, as described in Figure \ref{kesm15_periodic5}.
	\begin{figure}[h!]
	\begin{center}
	\includegraphics[scale=0.54]{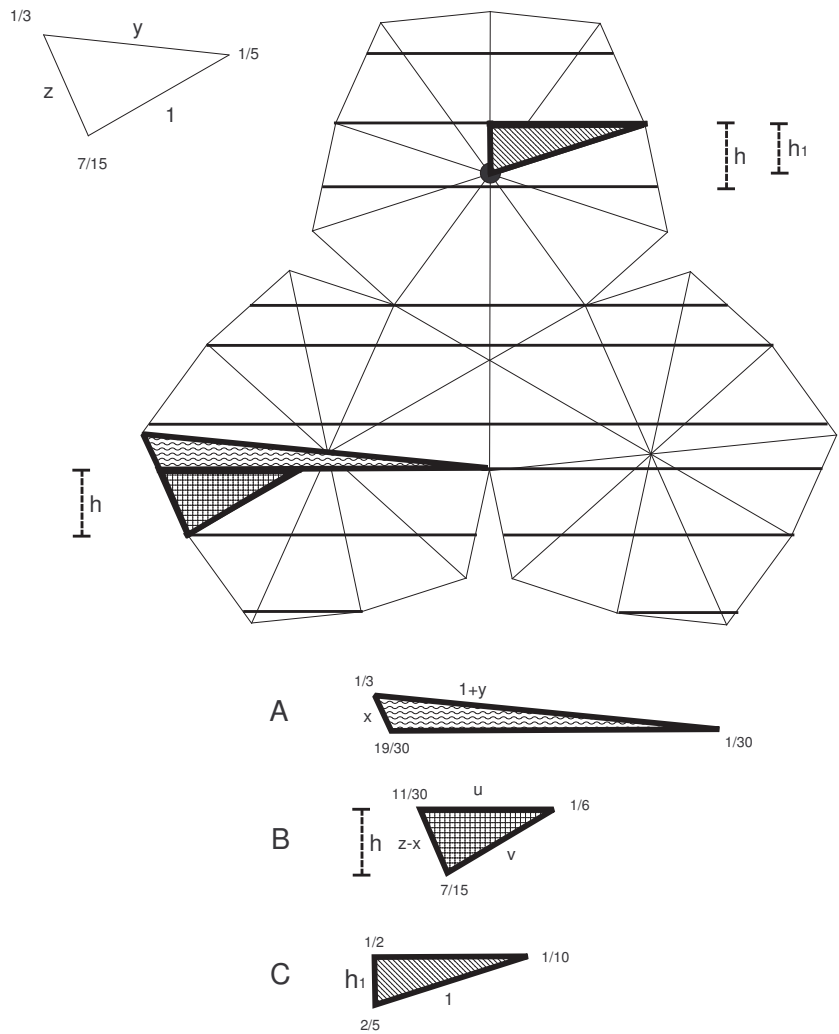}
	\caption{Cylinder decomposition in the horizontal direction of 
			the surface $M_P$ obtained from $P=\left(\frac{\pi}{5}, \frac{\pi}{3}, \frac{7\pi}{15}\right)$ - 
			calculations for the point corresponding to $\frac{\pi}{5}$.}
	\label{kesm15_periodic5}
	\end{center}
	\end{figure}

	Considering triangle $C$, we have: \ $h_1=\sin\left(\frac{\pi}{10}\right)$. We calculate $h$ by comparing two different formulas for 
	the area of triangle $B$ as follows:
	\begin{equation}
	\label{h}
	\frac{1}{2} \cdot h \cdot u = \frac{1}{2} \cdot v \cdot (z-x) \cdot \sin\left(\frac{7\pi}{15}\right)
		\quad \Longrightarrow \quad h=\frac{v\cdot (z-x) \cdot \sin\left(\frac{7\pi}{15}\right)}{u}
	\end{equation}
	
	The lengths of $z$ and $y$ are calculated from within the initial triangle 
	$\left(\frac{\pi}{5}, \frac{\pi}{3}, \frac{7\pi}{15}\right)$ as follows:
	\begin{equation}
	\label{y_z}
	z=\frac{\sin\left(\frac{\pi}{5}\right)}{\sin\left(\frac{\pi}{3}\right)} \quad ; \quad
		y=\frac{\sin\left(\frac{7\pi}{15}\right)}{\sin\left(\frac{\pi}{3}\right)}
	\end{equation}
	The length of $x$ is calculated from within triangle $A$ as follows: \;
	\[x=(1+y) \cdot \frac{\sin\left(\frac{\pi}{30}\right)}{\sin\left(\frac{19\pi}{30}\right)}\]
	The lengths of $v$ and $u$ are calculated from within triangle $B$:
	\[v=(z-x) \cdot \frac{\sin\left(\frac{11\pi}{30}\right)}{\sin\left(\frac{\pi}{6}\right)} \quad ; \quad
	u=(z-x) \cdot \frac{\sin\left(\frac{7\pi}{15}\right)}{\sin\left(\frac{\pi}{6}\right)}\]
	Substituting in (\ref{h}) we get:
	\[h=\frac{v}{u} \cdot (z-x) \cdot \sin\left(\frac{7\pi}{15}\right)=
		\frac{\sin\left(\frac{11\pi}{30}\right)}{\sin\left(\frac{\pi}{6}\right)} \cdot \frac{\sin\left(\frac{\pi}{6}\right)}{\sin\left(\frac{7\pi}{15}\right)} \cdot (z-x) \cdot \sin\left(\frac{7\pi}{15}\right)=
\sin\left(\frac{11\pi}{30}\right) \cdot (z-x)=\]
\[\sin\left(\frac{11\pi}{30}\right) \cdot \left(\frac{\sin\left(\frac{\pi}{5}\right)}{\sin\left(\frac{\pi}{3}\right)}-(1+a) \cdot \frac{\sin\left(\frac{\pi}{30}\right)}{\sin\left(\frac{19\pi}{30}\right)}\right)=
\sin\left(\frac{11\pi}{30}\right) \cdot \left[\frac{\sin\left(\frac{\pi}{5}\right)}{\sin\left(\frac{\pi}{3}\right)}-\left(1+\frac{\sin\left(\frac{7\pi}{15}\right)}
{\sin\left(\frac{\pi}{3}\right)}\right) \cdot \frac{\sin\left(\frac{\pi}{30}\right)}{\sin\left(\frac{19\pi}{30}\right)}\right]\]

	After simplifying the expressions we get: $\frac{h}{h_1}=\frac{1}{10}\left(5+\sqrt{75-30\sqrt{5}}\right) \notin \Q$.
	Therefore, according to Remark \ref{non-periodic}, the points correspond to the angle $\frac{\pi}{5}$ are non-periodic points.
\\ \\
	\textbf{Now, we will do the respective calculations for the points corresponding to the angle $\frac{\pi}{3}$}, 
				according to the markings in Figure \ref{kesm15_periodic3}:
	\begin{figure}[h!]
	\begin{center}
	\includegraphics[scale=0.5]{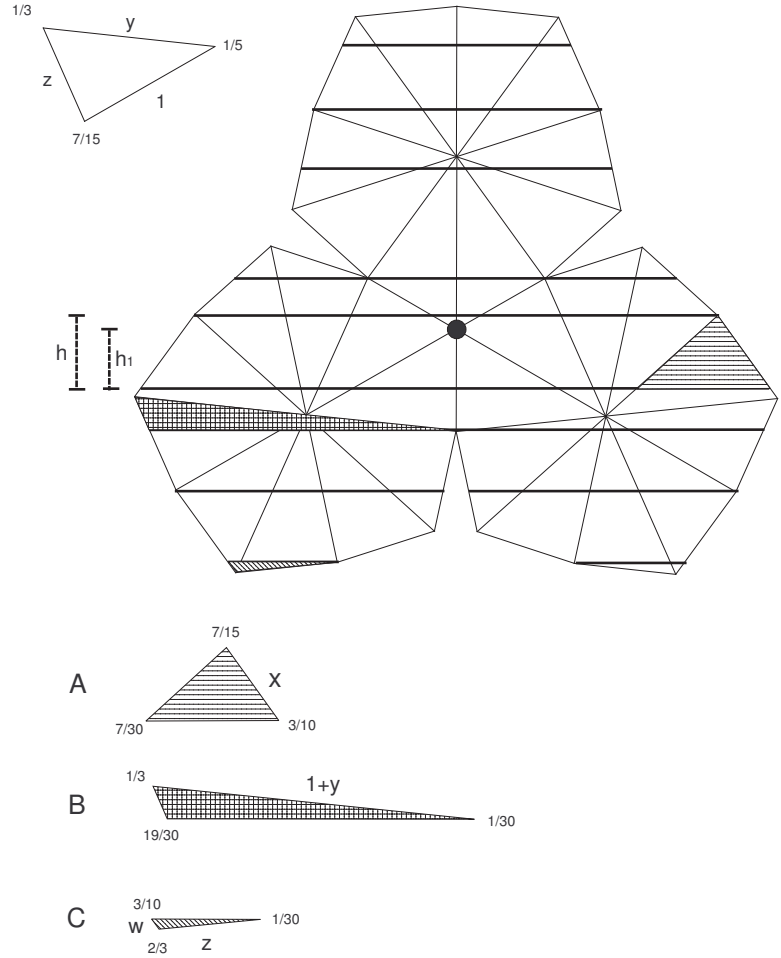}
	\caption{Cylinder decomposition in the horizontal direction of 
			the surface $M_P$ obtained from $P=\left(\frac{\pi}{5}, \frac{\pi}{3}, \frac{7\pi}{15}\right)$ - 
			calculations for the point corresponding to $\frac{\pi}{3}$.}.
	\label{kesm15_periodic3}
	\end{center}
	\end{figure}
First, we will calculate the height $h=h_A$ (where $h_A$ is the height of the triangle $A$):
\[h=h_A=x \cdot \sin\left(\frac{3\pi}{10}\right)\]
Where $x$ is calculated as follows:
\[x=z-w=\frac{\sin\left(\frac{\pi}{5}\right)}{\sin\left(\frac{\pi}{3}\right)}-
z\cdot \frac{\sin\left(\frac{\pi}{30}\right)}{\sin\left(\frac{3\pi}{10}\right)}\]
Second, we calculate the height $h_1$. As described in Figure \ref{kesm15_periodic3} ($y$ and $z$ as in (\ref{y_z})):
\[h_1=h_B+h_C \quad \text{where} \quad h_B=(y+1)\cdot \sin\left(\frac{\pi}{30}\right) \quad \text{and} \quad
h_C=z\cdot \sin\left(\frac{\pi}{30}\right)\]
Therefore, the wanted ratio is \; $\frac{h_1}{h}=\frac{h_B+h_C}{h}$,\; which simplifies to: \;$\frac{1}{2}(-1+\sqrt{5})\notin \Q$.
\item The surface $M_P$ obtained from the triangle $P$ with angles$\left(\frac{\pi}{2n}, \frac{\pi}{n}, \frac{(2n-3)\pi}{2n}\right)$ $n\geq 5$ 
	and odd, is illustrated in Figure \ref{surfaces_lemma} (surface 5). Since the points corresponding to the angle $\frac{\pi}{n}$
	are the centers of the regular n-gons, the same arguments as in the first part of the proof, imply that these points are non-periodic.
\end{enumerate}
\end{proof}

\newpage

\section{Proof of Theorem 4.5}

Some notations:
\begin{itemize}
\item $N_P$, $N_Q$ - The least common multiples of the denominators of the angles in $P$ and $Q$ respectively.
\item $M_P$, $M_Q$ - The surfaces obtained from billiard in $P$ and $Q$ respectively.
\end{itemize}

\begin{enumerate}
\item \textbf{\boldmath{$P$} is a regular polygon:}\\
	In this case all the vertices are singular points of the surface, and thus periodic. Since we want a single non-periodic branching 
	point, there is no appropriate cover.
\item  \textbf{\boldmath{$P$} is a right triangle with angles $\left(\frac{\pi}{n}, \frac{(n-2)\pi}{2n}, \frac{\pi}{2}\right)$, $n\geq 4$.}
	\begin{itemize}
	\item [a)] \textbf {If \boldmath{$n$} is even:}
		$M_P$ is the regular $n$-gon with parallel sides identified (see Figure \ref{octagon}).
		\begin{figure}[h!]
		\begin{center}
		\includegraphics[scale=0.31]{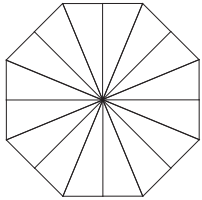}
		\caption{The octagon - $M_P$ for n=8}
		\label{octagon}			
		\end{center}
		\end{figure}
		All the points in $M_P$ corresponding to vertices in $P$ are periodic. This result is not new (see \cite{GHS}, Corollary 9), 
		but we will show another proof: All the points in $M_P$ corresponding to the vertices in $P$ are fixed points of a rotation by $\pi$.
		Hence, by Lemma \ref{involution}, these are all periodic points. Since we want a single non-periodic branch point, there is no 
		appropriate cover.
	\item [b)] \textbf {If \boldmath{$n$} is odd:}
		$M_P$ is the double regular n-gon with parallel sides identified (see Figure \ref{surfaces_lemma}, surface 1). $M_P$ has one singular 
		point corresponding to the angle $\frac{(n-2)\pi}{2n}$ in $P$. According to Lemma \ref{involution}, the midpoints of the sides of $M_P$
		(the points corresponding to the angle $\frac{\pi}{2}$) are periodic points, since they are fixed points of the rotation by $\pi$. 
		Following Lemma \ref{non-periodic_points}, the two centers of $M_P$, which correspond to the angle $\frac{\pi}{n}$, are non-periodic 
		points. Therefore, when we look for an appropriate cover, we actually look for a surface branched over one of the centers of $M_P$.

		($\star$) Notice that, considering Remark \ref{remark_fp_of_G_Q}, since these centers are not fixed points of the reflections with 
		respect to the sides of the regular n-gons, these reflections should not be in $G_Q$. Therefore, for an appropriate cover, 
		these sides must be internal in $Q$. 

		In the beginning, we will consider the first class of appropriate covers (as mentioned in page 12), i.e. a branched 
		cover \text{$\pi:M_Q \to M_P$}, where the branch locus is a single non-periodic point in $M_P$, meaning one of the centers of $M_P$. 
		First, according to Corollary \ref{not_appropriate}, for such a cover $N_Q$ must be odd. Second, we want that the singular point, 
		which is a periodic point, to be a regular point of the cover. Hence, following Lemma \ref{branching}, all the angles of the vertices 
		in $Q$, which correspond to the angle $\frac{(n-2)\pi}{2n}$ ($n$ is odd $\Longrightarrow (n-2,2n)=1$) must be 
		$k \cdot\frac{(n-2)\pi}{2n} \leq 2\pi$ with $k\mid 2n$ .
		\begin{itemize}
		\item $k \cdot\frac{(n-2)\pi}{2n} \leq 2\pi$ implies $k \in \{1,2,\ldots, 5\}$ for $n\geq 7$ and odd, 
			and $k \in \{1,2,\ldots, 6\}$ for $n=5$. 
		\item The requirement for odd $N_Q$ forces $k$ to be even. Hence $k \in \{2, 4\}$ for $n\geq 7$ and odd, and $k \in \{2, 4, 6\}$ 
			for $n=5$. 
		\item Finally, the requirement $k \mid 2n$ reduces the possibilities to $k=2$. 
		\end{itemize}
		($\clubsuit$) Consequently, each angle $2 \cdot \frac{n-2}{2n}\pi$ in $Q$ must be delimited by 2 external sides 
		(since we cannot expand it). 

		We will start constructing a polygon $Q$, under all the requirements above, step by step, as described in Figure \ref{steps}.
		\begin{figure}[h!]
		\begin{center}
		\includegraphics[scale=0.41]{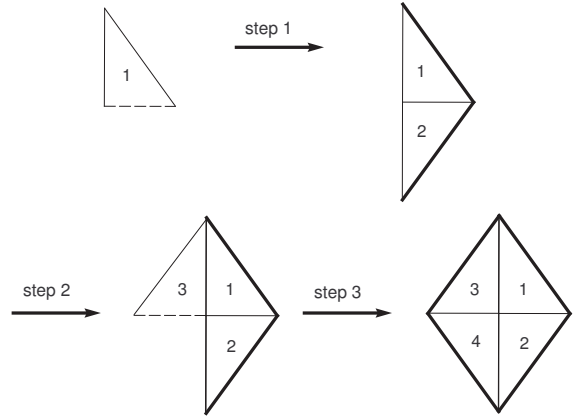}
		\caption{Steps of the construction of Q}
		\label{steps}
		\end{center}
		\end{figure}
		\textbf{Step 1:} Without loss of generality, we start with triangle number 1. The broken line indicates a side that must be internal 
			according to ($\star$). Therefore we add triangle number 2. Following ($\clubsuit$) we add the bold lines which indicate external 
			sides. Since the polygon in this step (triangles 1 and 2) determines a lattice surface (number 4 in the list), we need to enlarge 
			$Q$ and continuing to the next step. \quad
		\textbf{Step 2:} Without loss of generality we add triangle number 3. \quad
		\textbf{Step 3:} Again, as in step 1, the broken line indicates a side that has to be internal, therefore we add triangle number 4. 
			By ($\clubsuit$) we add the bold lines as in step 1. Now, since all sides are external, the construction is complete. \quad 
		According to Proposition \ref{branching}, the cover $\pi:M_Q \to M_P$, which corresponds to this polygon $Q$, is branched over two 
		different points corresponding to the centers of $M_P$. Therefore, there is no such appropriate cover.

		Up to this point, we have shown that there is no appropriate cover of the first class. Next, we will examine the possibility of finding 
		an appropriate cover of the second class. Therefore, we will look for:
		\begin{itemize}
		\item [\textbullet] A polygon $\overline{P}$ which is tiled by reflections of the triangle 
			$P=\left(\frac{\pi}{n}, \frac{(n-2)\pi}{2n}, \frac{\pi}{2}\right)$ $n\geq 5$ and odd, such that the covering map 
			$\pi:M_{\overline{P}} \to M_P$ is branched only over periodic points.
		\item [\textbullet] A polygon $Q$ tiled by $\overline{P}$ such that the covering map $\pi:M_Q \to M_{\overline{P}}$ \ 
			is branched over a single non-periodic point.
		\end{itemize}
		($\Diamondblack$) Notice that:
		\begin{itemize}
		\item If $\overline{P}$ has an angle $\alpha > \pi$, then any such polygon $Q$, must have this angle $\alpha$ as well. Therefore, If 
			$\overline{P}$ has an even angle greater than $\pi$, according to Corollary \ref{not_appropriate}, it will not yield an 
			appropriate cover.
		\item If $\overline{P}$ has two external sides with an angle greater than $\pi$ between them, then any such polygon $Q$, must have these 
			external sides as well. Therefore, If these sides are the sides of the regular n-gon, according to ($\star$) it will not yield an 
			appropriate cover.
		\end{itemize}

		As we mentioned before, except for the centers of $M_P$, all the points corresponding to the vertices in $P$, are periodic points
		of $M_P$. Following Proposition \ref{branching}, branching over a point corresponding to an angle $\frac{\pi}{2}$ implies 
		$\overline{P}$ has an angle $\frac{3\pi}{2}$. In that case, since $\frac{3\pi}{2}>\pi$, any polygon $Q$ which is tiled by 
		$\overline{P}$, will have this angle as well. Consequently, $N_Q$ will be even and by Corollary \ref{not_appropriate} it will 
		not be an appropriate cover. Therefore, if there exists an appropriate cover of the second class, then $M_{\overline{P}}$ will branch 
		only over the singular point of $M_P$.
		According to Proposition \ref{branching}, a branching over the singular point of $M_P$ will occur if and only if $\overline{P}$ 
		has an angle $k\cdot \frac{n-2}{2n}\pi<2\pi$ with $k\nmid 2n$. Therefore, we need to check the following possibilities for 
		appearance of an angle $k\cdot \frac{n-2}{2n}\pi$ in $\overline{P}$:
		\begin{itemize}
			\item $k \in \{3,4,6\}$  \; for $n=5$.
			\item $k \in \{3,4,5\}$  \; for $n=7$.
			\item $k \in \{4,5\}$  \; for $n=9$.
			\item $k \in \{3,4\}$  \; for $n\geq 11$ and $3\nmid n$.
			\item $k=4$  \; for $n\geq 11$ and $3\mid n$.
		\end{itemize}
		The case of $n=5$ will be examined in the sequel.

		For $n\geq 7$ and odd, the angle $k\cdot \frac{n-2}{2n}\pi$ with $k=3,5$ is even and greater than $\pi$. Hence, following 
		($\Diamondblack$) it will not give an appropriate cover. Therefore, for $n\geq 7$ and odd, it remains to check the possibility for 
		an angle $4\cdot \frac{n-2}{2n}\pi$ in $\overline{P}$. Such an angle can appear in $\overline{P}$ in two ways, as shown in Figure 
		\ref{4overlineP} (the bold lines indicate external sides for $\overline{P}$). Following ($\Diamondblack$), the left option in the 
		figure cannot occur since $4\cdot \frac{n-2}{2n}\pi > \pi$, and the bold lines are sides of the regular n-gon.
		\begin{figure}[h!]
		\begin{center}
		\includegraphics[scale=0.4]{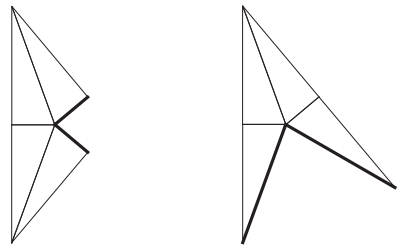}
		\caption{Two options to start the construction of $\overline{P}$	
					with the angle $4\cdot \frac{n-2}{2n}\pi$}				
		\label{4overlineP}
		\end{center}				
		\end{figure}
		
		It remains to check if it is possible to construct suitable $\overline{P}$ and $Q$, beginning with the right option in Figure 
		\ref{4overlineP}. We will start the construction under the requirements, step by step, as described in Figure 
		\ref{constructing_overlineP}.
		\begin{figure}[h!]
		\begin{center}
		\includegraphics[scale=0.52]{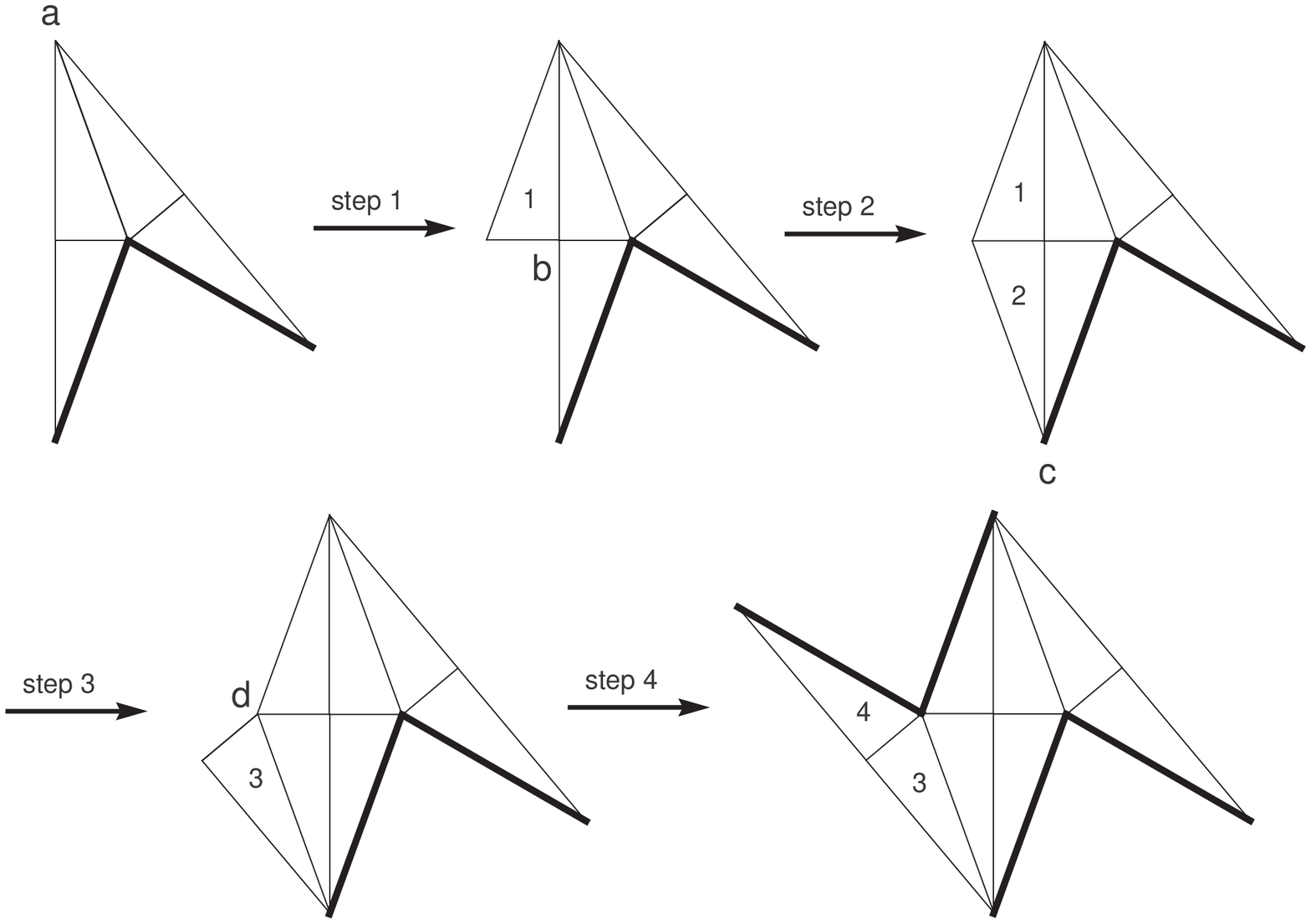}
		\caption{Constructing $\overline{P}$}
		\label{constructing_overlineP}
		\end{center}
		\end{figure}

		\textbf{Step 1:} According to Proposition \ref{branching}, in order to prevent a branching over a non-periodic 
			point, we need to fix the angle at the vertex $a$. Therefore we add triangle number 1. \quad
		\textbf{Step 2:} Following ($\Diamondblack$), since the angle of the vertex $b$ is even and greater than $\pi$, 
			we need to fix that angle by adding triangle number 2. \quad 
		\textbf{Step 3:} As in step 1, we fix the angle at vertex $c$ by adding triangle number 3. \quad 
		\textbf{Step 4:} Since the angle $3 \cdot \frac{n-2}{2n}$ of vertex $d$ is even and greater than $\pi$, according to 
			($\Diamondblack$), we have to enlarge it to $4 \cdot \frac{n-2}{2n}$ by adding triangle number 4.

		According to Propositions \ref{branching} and \ref{commensurable}, the polygon in the last step has the lattice property if and only 
		if $3\mid n$. In that case, we will try to construct a suitable polygon $Q$ tiled by this polygon. Without loss of generality, we 
		reflect $\overline{P}$ in the broken line as described in Figure \ref{overlineP_to_Q}. Consequently, according to Remark 
		\ref{gluing}, $Q$ have 2 singular vertices, $v$ and $u$, with angles $\frac{2\pi}{n}$ and $\frac{6\pi}{n}$ respectively. 
		These vertices are corresponding to the two centers in $M_P$. Since we cannot fix these angles by reflecting $\overline{P}$ 
		again (as can be shown in the figure), by Proposition \ref{branching}, the respective cover will have two branch points corresponding 
		to the two centers of $M_P$. Hence, it will not be an appropriate cover. Therefore, we will try expand the construction of 
		$\overline{P}$ after the last step in order to find an appropriate cover.
		\begin{figure}[h!]
		\begin{center}
		\includegraphics[scale=0.49]{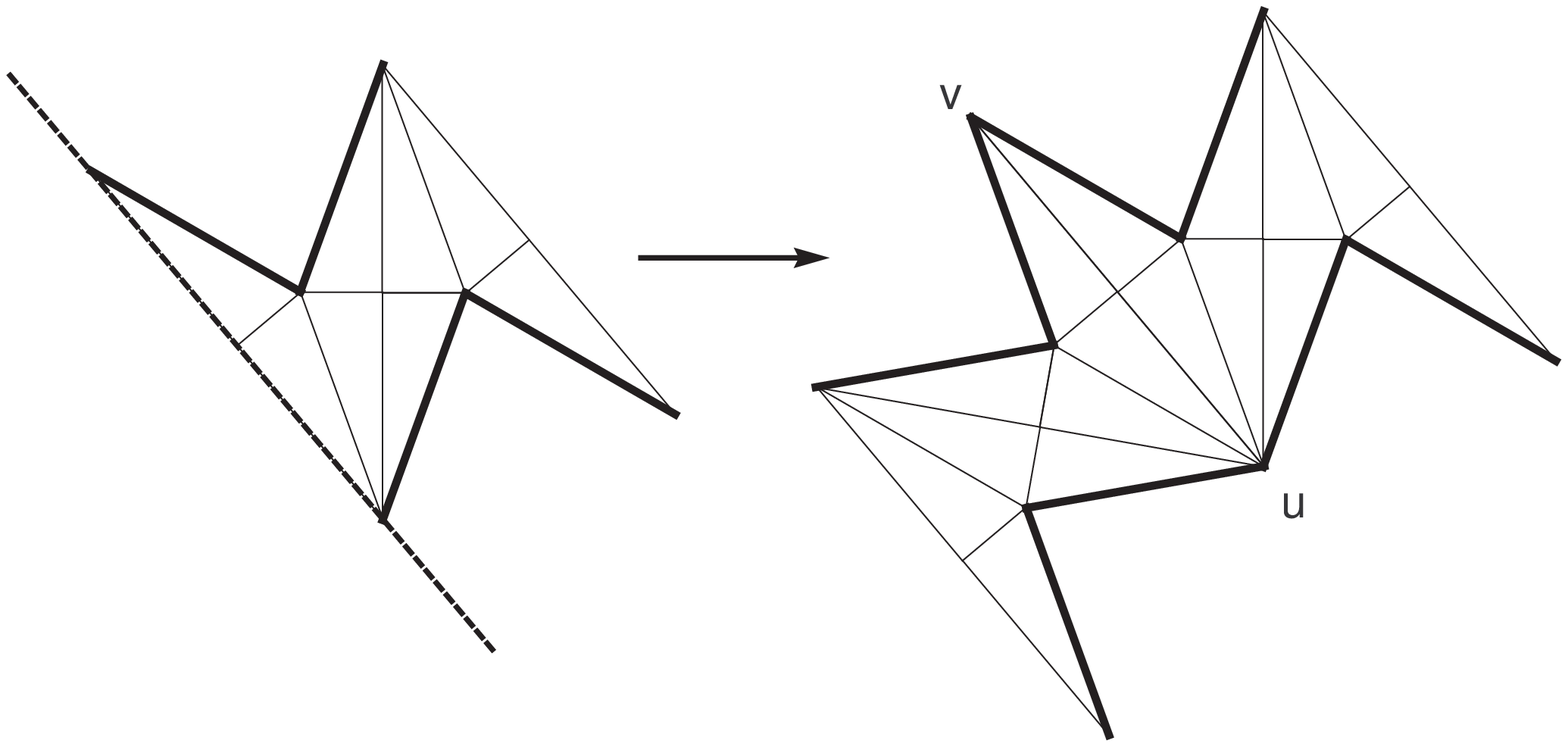}
		\caption{$Q$ tiled by reflections of $\overline{P}$}				
		\label{overlineP_to_Q}
		\end{center}				
		\end{figure}

		The continuing of the construction is described in Figure \ref{continuing_overlineP}.
		\textbf{Step 1:} Without loss of generality, we add triangle number 5. Following the arguments in ($\Diamondblack$), 
			we must add triangle number 6. \quad
		\textbf{Step 2:} According to Remark \ref{gluing}, this polygon has 2 singular vertices corresponding to the two centers in $M_P$, 
			$a$ and $b$, with angles $\frac{2\pi}{n}$ and $\frac{4\pi}{n}$ respectively. 	Therefore, according to Proposition \ref{branching}, 
			a surface obtained from this polygon will have 2 different non-periodic branch points corresponding to the centers of $M_P$. 
			Hence, we must fix these angles by adding triangles number 7 and 8. By ($\Diamondblack$), in that case, it cannot yield an 
			appropriate cover.
		\begin{figure}[h!]
		\begin{center} 
		\includegraphics[scale=0.76]{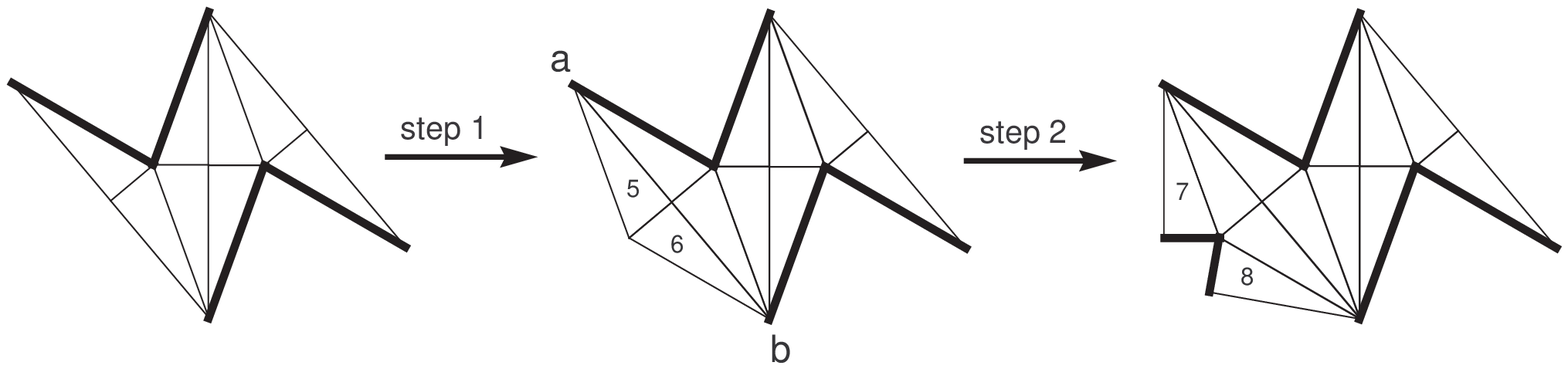}
		\caption{Continuing the construction of $\overline{P}$}
		\label{continuing_overlineP}
		\end{center}				
		\end{figure}
		
		Up to this point we have shown that for any $n \geq 7$ and odd, there is no appropriate cover. It remains to check the case of $n=5$.
		According to Proposition \ref{branching}, since we want $M_{\overline{P}}$ to branch over periodic points in $M_P$, all the vertices 
		in $\overline{P}$ corresponding to the angle $\frac{\pi}{5}$, have to appear with angles $\frac{\pi}{5}$, $\pi$ or $2\pi$. ($\spadesuit$)

		Since we are interested in a polygon $Q$, tiled by $\overline{P}$, such that $M_Q$ is branched over a single non-periodic point, 
		we must have a vertex with angle $\frac{\pi}{5}$ in $\overline{P}$. Hence, we start the construction with triangle number 1 with two 
		external sides (as described in Figure \ref{pentagon}). The following steps were taken:

		\textbf{Step 1:} The only way to continue is by adding triangle number 2 with an external side. This side must be external, otherwise,
			adding a triangle to the left of triangle 2, leads to an angle of $\frac{3\pi}{2}$. As explained before, such an angle cannot appear 
			in $\overline{P}$. Since the sides of triangle number 1 are external, we cannot enlarge it to $2\pi$. \quad
			The polygon in this step (triangles 1 and 2) determines a lattice surface which we treated in the first part of this section 
			(case 2b, appropriate cover of the first class). Hence we need to enlarge $Q$ and continue to the next step. \quad
 		\textbf{Step 2:} Adding triangle number 3. \quad
		\textbf{Step 3:} Following ($\spadesuit$), we need to fix the angle $\frac{2\pi}{5}$ of vertex $a$. The only way to do it, 
			is by enlarging it to $\pi$ by adding triangles 4, 5 and 6. \quad
		\begin{figure}[h!]
		\begin{center}
		\includegraphics[scale=0.5]{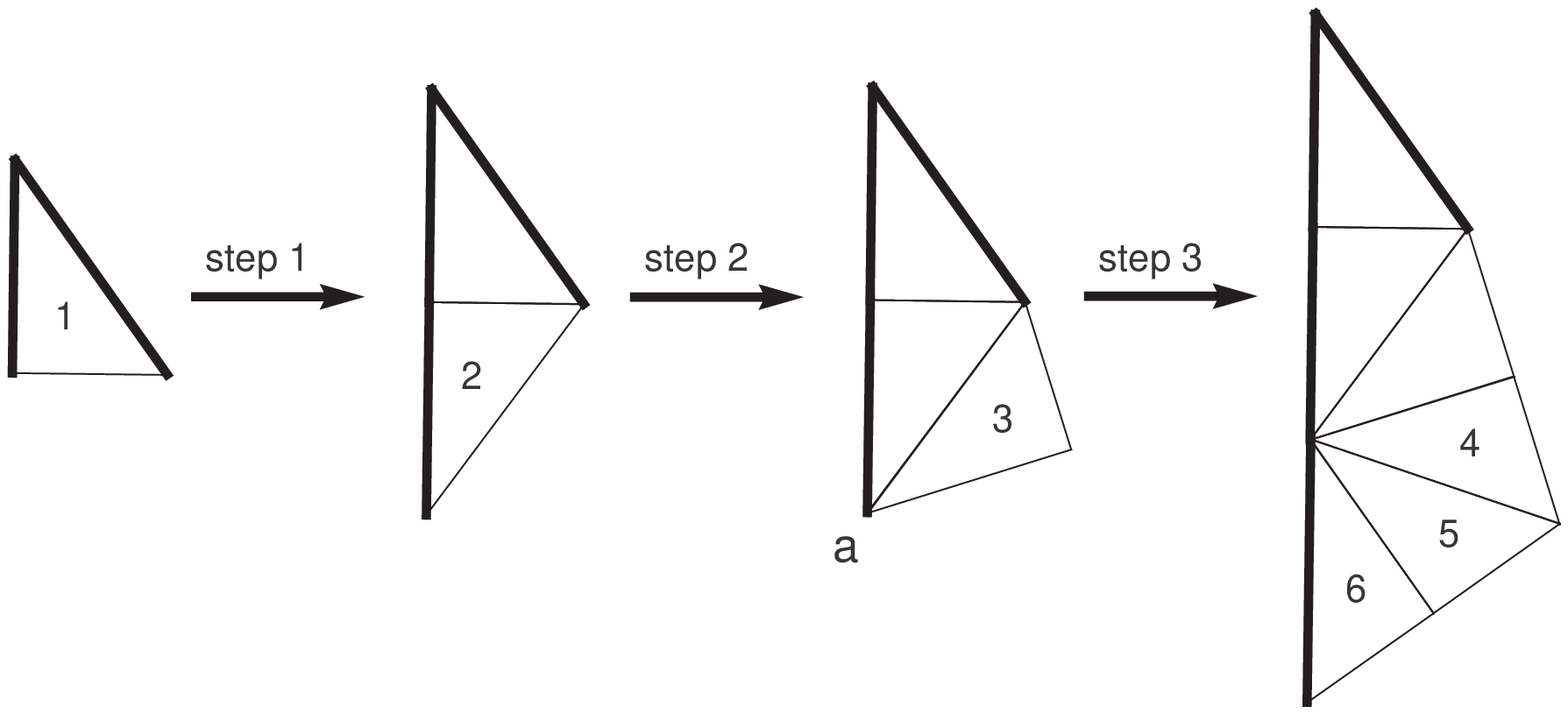}
		\caption{The constructing of $\overline{P}$}
		\label{pentagon}
		\end{center}				
		\end{figure}

		Consequently, any suitable $\overline{P}$ must contain the last shape in Figure \ref{pentagon}. We will show that for any such 
		polygon $\overline{P}$, we cannot find a polygon $Q$ tiled by $\overline{P}$ that would yield an appropriate cover. 

		We have started the construction of $\overline{P}$ with an angle $\frac{\pi}{5}$ in order to multiply this angle to be
		$\frac{k\pi}{5}$ with $k>1$ and $(k,5)= 1$ in $Q$. That is, for having the wanted branching of the cover 
		$\pi:M_ Q\to M_{\overline{P}}$. Hence, to find a suitable polygon $Q$, we have to reflect with respect to at least one of the
		sides of that angle. We describes two possible reflections in Figure \ref{Q_from_overlineP}. The first reflection 
		cannot occur, since in this case we get two external sides, $A$ and $B$, that must be internal according to ($\star$).
		\begin{figure}[h!]
		\begin{center}
		\includegraphics[scale=0.5]{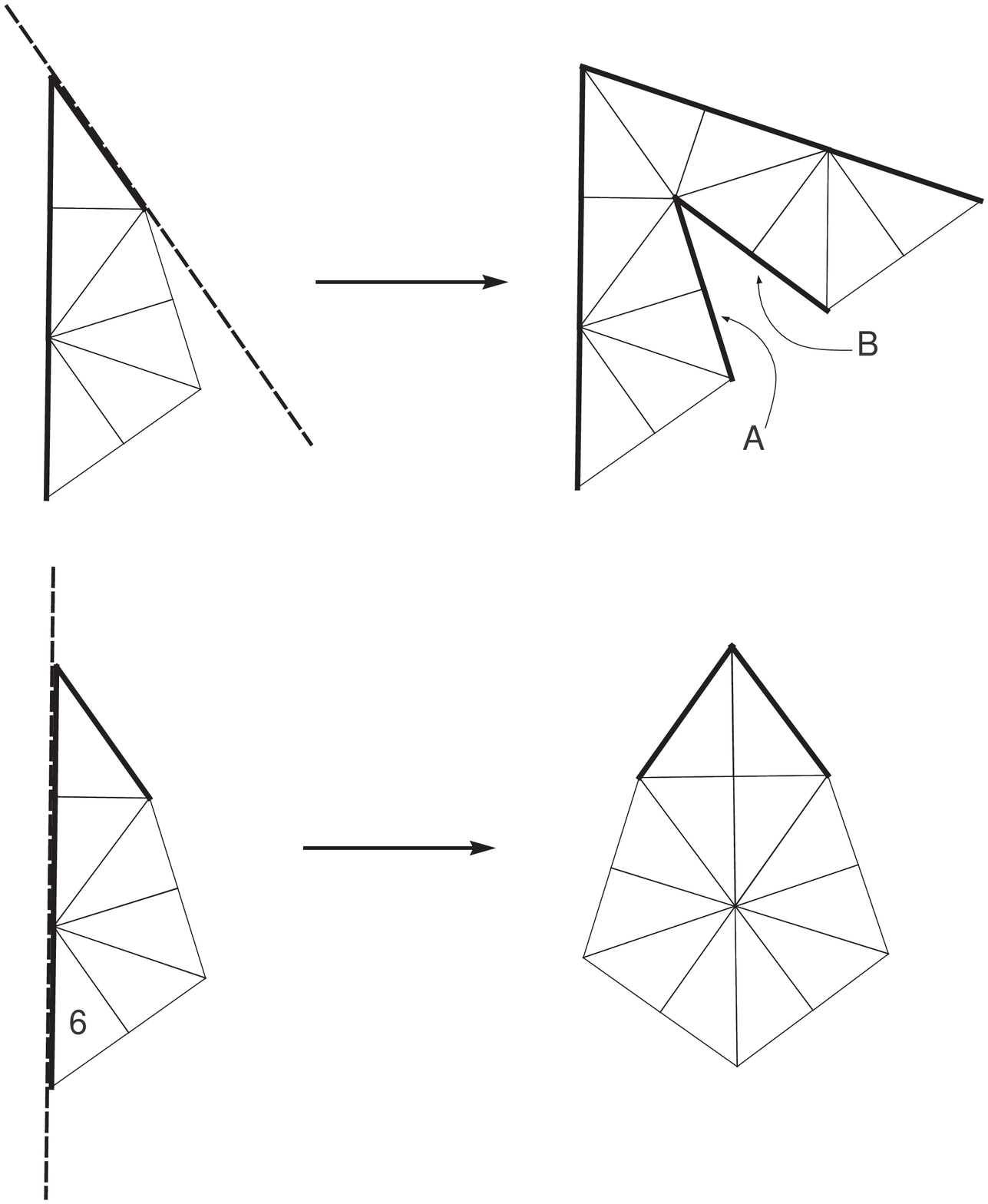}
		\caption{Constructing from $\overline{P}$ for n=5}
		\label{Q_from_overlineP}
		\end{center}				
		\end{figure}

		To show that the second reflection cannot occur as well, we will first explain why the side of triangle number 6 must be
		external in $\overline{P}$. The reflection with respect to the vertical direction multiplies the angle of the lower vertex 
		of triangle 6. Therefore, if this side is internal, the angle of vertex $a$ in $\overline{P}$ must be $2\cdot\frac{3\pi}{10}$ 
		or $3 \cdot\frac{3\pi}{10}$. We will examine these options as described in Figure \ref{external6}.
		\begin{itemize}
		\item If the angle is $2\cdot\frac{3\pi}{10}$, according to Proposition \ref{branching} with $\overline{P}$ playing the role of P 
			and $m_0=3$, $n_0=5$, the reflection with respect to the vertical direction will cause a branching of the cover 
			$\pi:M_Q \to M_{\overline{P}}$ over the singular point that corresponds to this vertex. Therefore, it will not yield an 
			appropriate cover.
		\item If the angle is $3\cdot\frac{3\pi}{10}$, such a reflection will cause two external sides of the regular n-gon. 
			Following ($\star$), it will not give an appropriate cover.
		\end{itemize}
		\begin{figure}[h!]
		\begin{center}
		\includegraphics[scale=0.55]{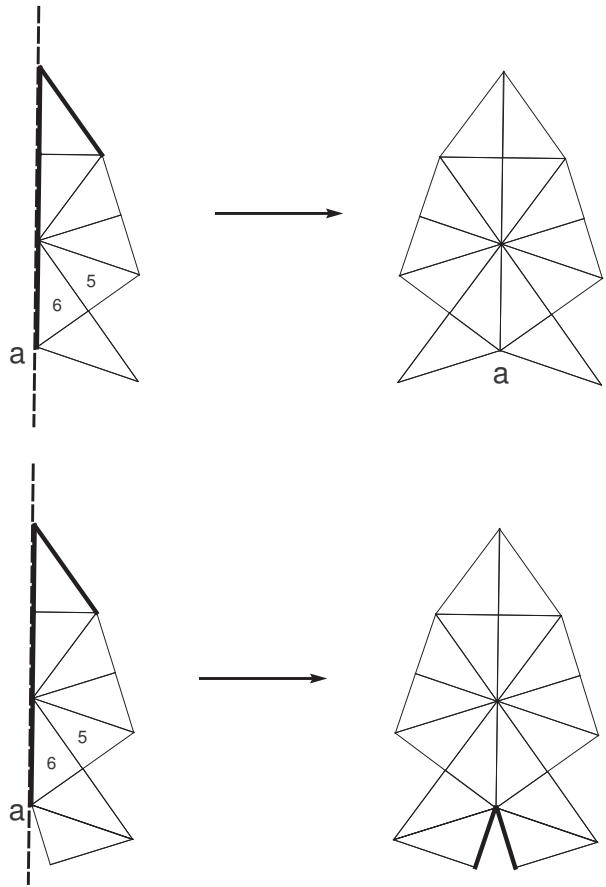}
		\caption{The side in triangle number 6 has to be 
				external in $\overline{P}$}
		\label{external6}
		\end{center}				
		\end{figure}
	
		Hence, the angle of vertex $a$ must be $\frac{3\pi}{10}$, i.e. the side of triangle number 6 must be external in $\overline{P}$. 
		This implies an external side in triangle number 5 as well, as illustrated in the left of Figure \ref{right5_buildingQ}.

		Up to this point, we have shown that $\overline{P}$ must contain the polygon in the left of Figure \ref{right5_buildingQ}.
		We also showed that in order to find a suitable polygon $Q$, we must reflect $\overline{P}$ with respect to line $A$. 
		According to ($\star$), we must to reflect with respect to line $B$ as well. These two reflections multiply the angle of vertex 
		$a$ by 3. By Proposition \ref{branching}, it will cause a branching over the singular point in $M_P$. Therefore, also 
		for $n=5$, we cannot find an appropriate cover.

		\begin{figure}[h!]
		\begin{center}
		\includegraphics[scale=0.7]{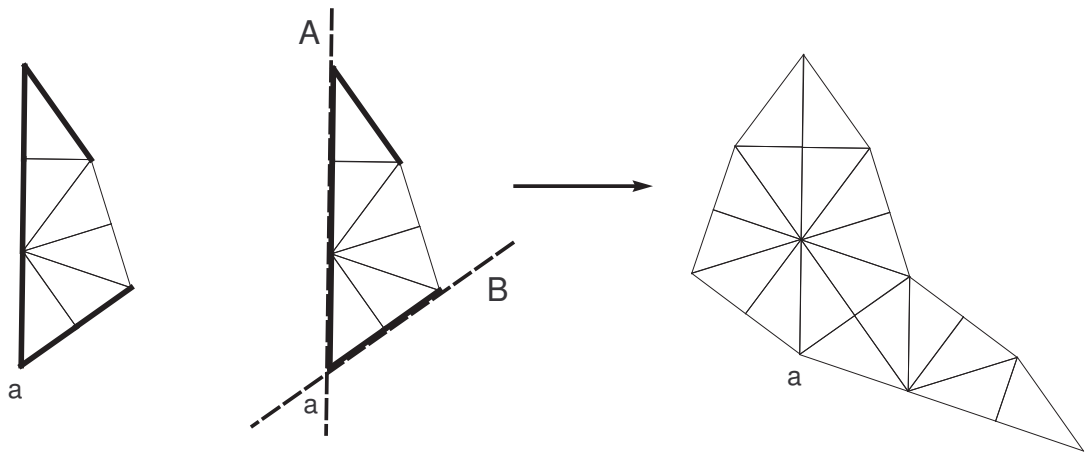}
		\caption{There is no appropriate polygon}
		\label{right5_buildingQ}
		\end{center}				
		\end{figure}
\end{itemize}

\item \textbf{\boldmath{$P$} is an acute isoceles triangle with angles 
		$\left(\frac{(n-1)\pi}{2n}, \frac{(n-1)\pi}{2n}, \frac{\pi}{n}\right)$,
		$n\geq 3$.}\\
		$M_P$ is the regular $2n$-gon with parallel sides identified. As in case 2(a), all the points corresponding to vertices 
		in $P$ are periodic. Hence, there is no appropriate cover.

\item \textbf{\boldmath{$P$} is an obtuse isosceles triangle with angles
	$\left(\frac{\pi}{n}, \frac{\pi}{n}, \frac{(n-2)\pi}{n}\right)$, $n\geq 5$.}
	\begin{itemize}
	\item [a)] \textbf {If \boldmath{$n$} is even:}
		The surface obtained from the billiard in $P$ is a regular double cover (without branching) of the surface in 2(a), i.e.
		double $2n$-gon (see Figure \ref{double_octagon}). Therefore, if there is an appropriate branched cover of $M_P$, 
		it will be an appropriate one for the surface in 2(a), for which we proved there is no such cover.
		\begin{figure}[h!]
		\begin{center}
		\includegraphics[scale=0.38]{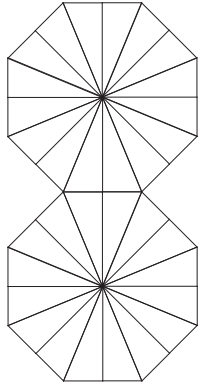}
		\caption{The double octagon, $M_P$ for n=8}				
		\label{double_octagon}
		\end{center}				
		\end{figure}

	\item [b)] \textbf {If \boldmath{$n$} is odd:}
		The surface obtained from the billiard in $P$ is the same surface as in 2(b). Therefore there is no appropriate cover.
	\end{itemize}

\item \begin{itemize}
	\item [a)] \textbf{\boldmath{$P$} is the acute scalene triangle with angles 
	$\left(\frac{\pi}{4}, \frac{\pi}{3}, \frac{5\pi}{12}\right)$.}\\
	Denote the vertices of $P$ by $a$, $b$ and $c$ as in the proof of Lemma \ref{non-periodic_points}. 
	$M_P$ is a surface of genus 3 with one singular point corresponding to the vertex $c$ in $P$ 
	(see Figure \ref{surfaces_lemma}, surface 2). Therefore, if there exists an appropriate cover, the branching must be over a regular point 
	corresponding to a vertex $a$ or $b$. According to Proposition \ref{branching}, if the branch point corresponds to a vertex $b$, 
	then $N_Q$ is even. In that case, according to Corollary \ref{not_appropriate}, we cannot find an appropriate cover. 
	Therefore, we should examine only the first option: The branch point corresponding to a vertex $a$. In this case, Lemma \ref{fp_of_G_Q} 
	implies that $G_Q$ must be isomorphic to $D_3$ (the only dihedral subgroup of $G_P$ that fixes these points). Therefore, all the 
	vertices in $Q$ must be of the form $\frac{k\pi}{3}$. In particular, this requirement implies that every angle of a vertex $c$ must be 
	multiplied by 4, and every angle of a vertex $b$ must be canceled, i.e. multiplied by 4 or 8. ($\star$)
	
	There are 2 possibilities for gluing together 4 vertices of type $c$ by reflections (see Figure \ref{four_c}). The bold lines 
	indicate exterior sides of the polygon (since we cannot expand the angle according to ($\star$)).
	\begin{figure}[h!]
	\begin{center}\qquad
	\includegraphics[scale=0.45]{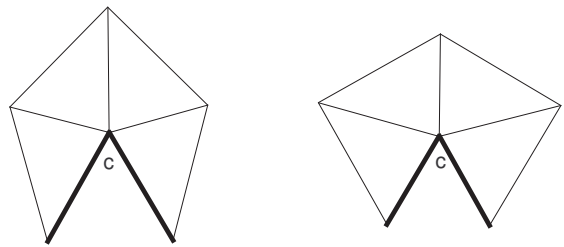}
	\caption{Two options gluing together 4 triangles in vertex $c$}
	\label{four_c}
	\end{center}
	\end{figure}

	\textbf{First we will show that the left option in Figure \ref{four_c} cannot be contained in $Q$}.
	We will start constructing $Q$ with these four triangles, step by step, under above requirements.
	\begin{figure}[h!]
	\begin{center}
	\includegraphics[scale=0.67]{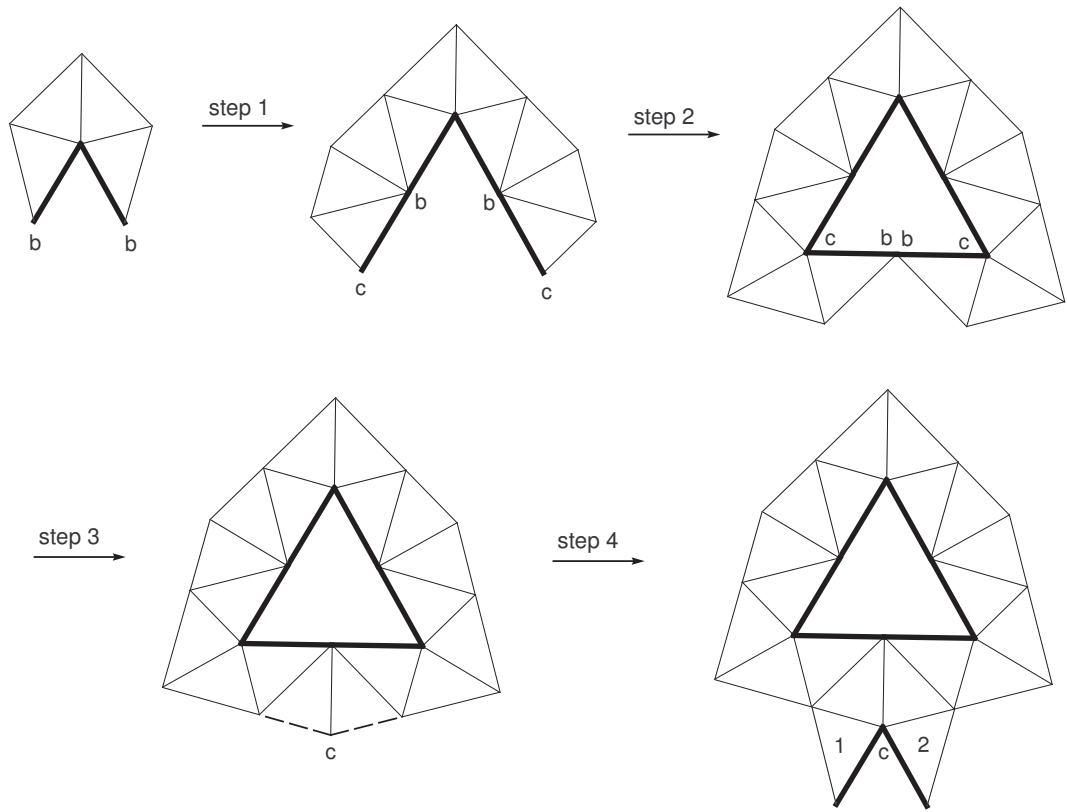}
	\caption{Steps of the construction of $Q$.}
	\label{building_Q}
	\end{center}
	\end{figure}

	\textbf{Step 1}: Canceling the angles of vertices $b$.\quad
	\textbf{Step 2}: Multiplying the angles of vertices $c$ by 4.\quad
	\textbf{Step 3}: Canceling the angles of the vertices $b$.\quad
	\textbf{Step 4}: The broken lines form an even angle with the exterior bold line. Following Corollary \ref{not_appropriate}, 
		these sides must be internal. Therefore, we have to add triangles number 1 and 2.

	At this point, the polygon contains another 4 triangles as those we began with (at the bottom). Retracing the same steps, we will get 
	infinitely many triangles at $Q$. Since we are interested in a finite cover of $M_P$, this cannot yield an appropriate cover. Therefore, 
	the left option in Figure \ref{four_c} is not contained in $Q$.

	\textbf{Now, we will show that the right option in Figure \ref{four_c}, cannot be contained in $Q$ as well}. Again, we will begin constructing $Q$ 
	with these four triangles. The following steps describe the construction of $Q$ (see Figure \ref{building_Q2}).
	\begin{figure}[h!]
	\begin{center}
	\includegraphics[scale=0.55]{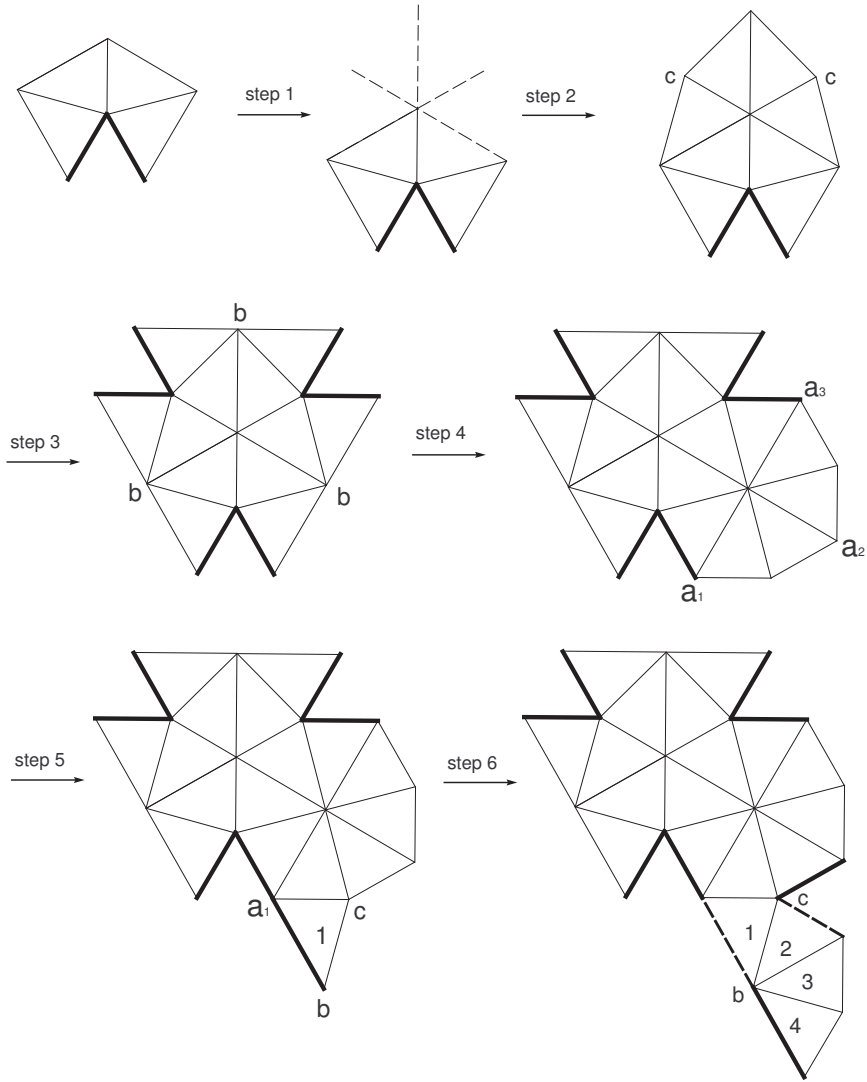}
	\caption{Steps of the construction of $Q$.}
	\label{building_Q2}
	\end{center}
	\end{figure}

	\textbf{Steps 1 and 2}: All the broken lines form an even angle with one of the exterior bold lines. Therefore, by Corollary 
		\ref{not_appropriate}, these sides must be internal.\quad
	\textbf{Step 3}: Adding triangles in order to multiply the angles of vertices $c$ by 4. This is the only way to do that, since 
		we have shown that the left option cannot be contained in $Q$.\quad 
		At that point, according to Propositions \ref{branching} and \ref{commensurable}, this polygon determines a lattice surface, 
		hence we need to enlarge $Q$. \quad	
	\textbf{Step 4}: Without loss of generality, we added triangles at one of the vertices $b$. As we required before, at each 
		vertex $b$ there are 4 or 8 triangles glued together. Therefore we added 4 triangles. \quad
	\textbf{Step 5}: Consequently, we have 3 vertices of type $a$, in each we have an angle $\frac{2\pi}{3}$. The corresponding points 
		to these vertices, are different points in $M_P$. According to Proposition \ref{branching}, since we are interested in covers with a 
		single branch point, we have to correct at least two of the angles of these vertices to $\pi$. Without loss of generality, we start with 
		vertex $a_1$ by adding triangle number 1. \quad
	\textbf{Step 6}: Canceling the angle of vertex $b$ by adding triangles number 2, 3 and 4.\quad
	In the last step of the construction, we got 2 broken bold lines that form an even angle of $\frac{\pi}{6}$ between them. Hence, 
	according to Corollary \ref{not_appropriate} it will not yield an appropriate cover.

	Up to this point we have shown that there is no appropriate cover of the first class (as mentioned in page 12). Now, we will examine the 
	possibility of finding an appropriate cover of the second class. According to Proposition \ref{branching}, if the cover 
	$\pi:M_{\overline{P}} \to M_Q$ is branched over a point corresponding to a vertex $b$ in $P$, then $\overline{P}$ has
	one of the following angles: $\frac{3\pi}{4}, \frac{5\pi}{4}, \frac{6\pi}{4}=\frac{3\pi}{2}, \frac{7\pi}{4}$.
	For each of these possibilities, any polygon $Q$ which is tiled by reflections of $\overline{P}$, must have an even angle. 
	Following Corollary \ref{not_appropriate}, it will not yield an appropriate cover.
	Therefore, by Lemma \ref{non-periodic_points}, it remains to check the possibility for an appropriate cover, when the cover 
	$\pi:M_{\overline{P}} \to M_Q$ is branched over the singular point. According to Proposition \ref{branching}, such a branching will 
	occur if and only if there will be an angle $k\cdot \frac{5\pi}{12}$ with $k\nmid 12$. Since for any such $k$, 
	$k \cdot \frac{5\pi}{12} > 2\pi$, there is no such a cover.

	\item [b)] \textbf{\boldmath{$P$} is the acute scalene triangle 
	with angles $\left(\frac{\pi}{5}, \frac{\pi}{3}, \frac{7\pi}{15}\right)$.}\\
	Denote the vertices of $P$ by $a$, $b$ and $c$ as in the proof of Lemma \ref{non-periodic_points}. 
	$M_P$ is a surface of genus 4 with one singular point corresponding to vertex $b$ (see Figure \ref{surfaces_lemma}, surface 4). 
	Therefore, if there exists an appropriate cover, the branching must be over a regular point corresponding to a vertex $a$ or $c$. 
	In these cases, Lemma \ref{fp_of_G_Q} implies that $G_Q$ must be isomorphic to $D_3$ or $D_5$ respectively (the dihedral subgroups 
	of $G_P$ that fix the corresponding points in $M_P$). The first option cannot occur since $G_Q\cong D_3$ implies that all the 
	denominators of the angles in $Q$ are 3, and $5 \cdot \frac{7\pi}{15} > 2\pi$. Therefore we should examine only the second option: 
	The branch point corresponding to a vertex $c$. In that case $G_Q \cong D_5$. Therefore, all the vertices in $Q$ must be of the form 
	$\frac{k\pi}{5}$. In particular, this requirement implies that every angle of a vertex $b$ must be multiplied by 3, and every angle of 
	a vertex $a$ must be cancelled, i.e. multiplied by 3 or 6. ($\star$)

	We will start constructing $Q$ with 3 triangles glued together in vertex $b$. The following steps describe the construction of $Q$
	(see Figure \ref{kesm15_building}). The bold lines indicate exterior sides for the polygon. These bold lines are added when 
	we cannot expand the angle between them.
	\begin{figure}[h!]
	\begin{center}
	\includegraphics[scale=0.58]{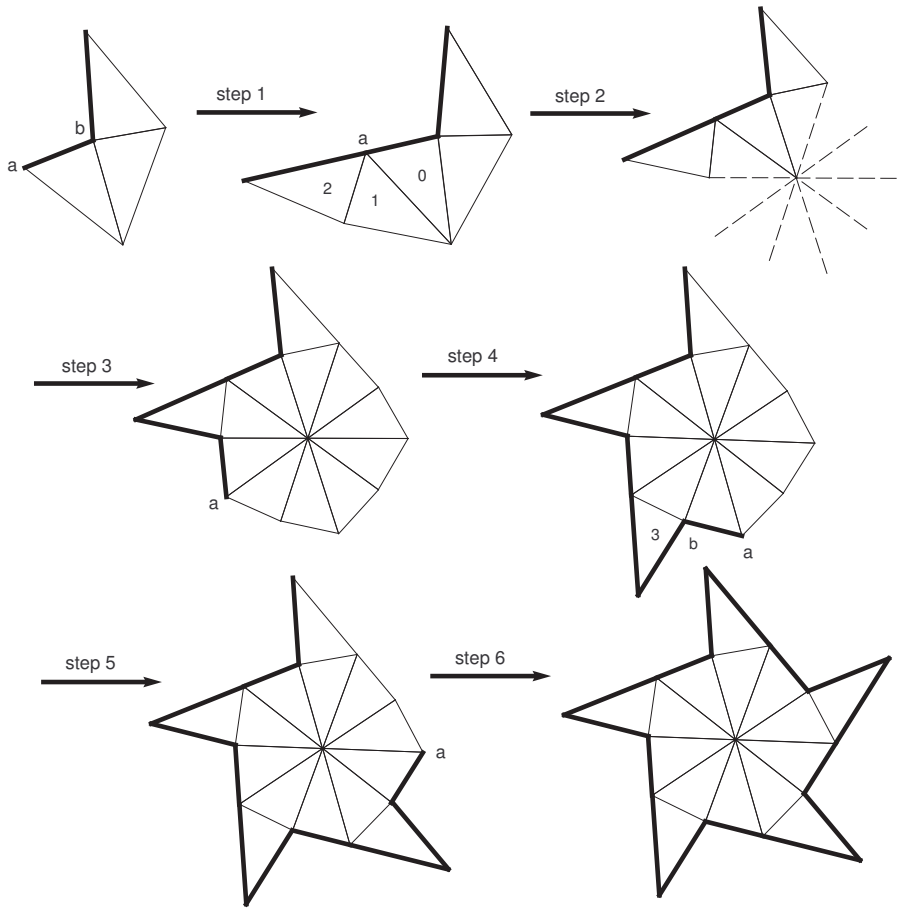}
	\caption{Steps of the construction of $Q$}
	\label{kesm15_building}
	\end{center}
	\end{figure}

	\textbf{Step 1:} Cancelling the angle of vertex $a$ by adding triangles 1 and 2. We cannot enlarge the angle of vertex $a$
		to $2\pi$ because of the existence of an exterior side in triangle $0$. Hence, we have an exterior side in triangle number 2.\quad
	\textbf{Steps 2 and 3:} The broken lines form an angle of $\frac{k\pi}{15}$, $k\nmid 15$ with one of the bold lines.
		Therefore, these sides must be internal, otherwise, $G_Q$ (the group generated by reflections of the sides of $Q$) will be the same 
		as $G_P$ ($D_{15}$) and not as we required ($D_5$). Moreover, the side touching vertex a is external according to $\star$. \quad
	\textbf{Step 4:} Fixing the angle $\frac{2\pi}{3}$ of vertex $a$ to be $\pi$ by adding triangle number 3. The new bold 
		lines are added since we cannot enlarge the angle of vertex $b$ according to ($\star$).\quad
	\textbf{Steps 5+6:} Repeating step 4 - adding triangles to fix the angles of the vertices $a$ to be $\pi$.

	At that point, the construction of $Q$ is complete, since all of its sides are exterior. According to Propositions \ref{branching} 
	and \ref{commensurable}, $Q$ determines a lattice surface. Therefore, it is not an appropriate cover of $M_P$.

	Up to this point we have shown that there is no appropriate cover of the first class (as mentioned in page 12). 
	Now, we will examine the possibility of finding an appropriate cover of the second class. 

	According to Lemma \ref{non-periodic_points}, the points corresponding to vertices $a$ and $c$ are non-periodic.
	Since the branch points of the cover can arise only from the vertices of $P$, the only periodic point in $M_P$, which $M_{\overline{P}}$ 
	can be branched over, is the singular point corresponding to vertex $c$ in $P$. Following Proposition \ref{branching}, the cover 
	$\pi:M_{\overline{P}} \to M_Q$ will branch over the singular point of $M_P$ if and only if $\overline{P}$ has an angle 
	$k\cdot \frac{7\pi}{15}<2\pi$ with $k\nmid 15$. Therefore, for such an appropriate cover, $\overline{P}$ must have an angle 
	$2\cdot \frac{7\pi}{15}$ or $4\cdot \frac{7\pi}{15}$. Since $4\cdot \frac{7\pi}{15} > \pi$, any polygon $Q$ tiled by $\overline{P}$
	will have this angle as well. In that case, $N_Q=15$, and $G_Q$ will not be isomorphic to $D_5$ as required. Hence, if there is an 
	appropriate cover of the second class, $\overline{P}$ must have the angle $2\cdot \frac{7\pi}{15}$. But, in that case, any polygon 
	$Q$ tiled by $\overline{P}$, will have an angle of the form $k\cdot 2\cdot \frac{7\pi}{15} < 2\pi$, i.e. $Q$ will have an angle
	$2\cdot \frac{7\pi}{15}$ or $4\cdot \frac{7\pi}{15}$. As in the previous case, $N_Q=15$ so it will not yield an appropriate cover.

	\item [c)] \textbf{\boldmath{$P$} is the acute scalene triangle with angles
	$\left(\frac{2\pi}{9}, \frac{\pi}{3}, \frac{4\pi}{9}\right)$.}\\
	Denote the vertices of $P$ by $a$, $b$ and $c$ as in the proof of Lemma \ref{non-periodic_points}. $M_P$ is a surface of 
	genus 3 with 2 singular points corresponding to vertices $a$ and $c$ (see Figure \ref{surfaces_lemma}, surface 3). Therefore, if there 
	is an appropriate cover it must be branched 	over a regular point corresponding to a vertex $b$. By Lemma \ref{fp_of_G_Q}, in that case, 
	$G_Q$ must be isomorphic to $D_3$ (the dihedral subgroup of $G_P$ which fixes the points in $M_P$ corresponding to vertex $b$). 
	This implies that all the vertices in $Q$ must be of the form $\frac{k\pi}{3}$.
	In particular, every angle of a vertex $c$ in $Q$ must be multiplied by 3. For a vertex $a$, we can achieve this form if we multiplied 
	the angle by 3 or 6. According to Proposition \ref{branching}, multiplying by 6 will lead to a branching over a singular point (periodic).
	Hence, every angle of vertex $a$ must be multiplied by 3 in $Q$. ($\star$)

	We will start constructing $Q$ with 3 triangles glued together in vertex $c$. As before, the 2 bold lines indicate external sides 
	(occur when we cannot enlarge the angle at the vertex).The following steps describe the construction of $Q$ (see Figure 
	\ref{kesm9_building}).
	\begin{figure}[h!]
	\begin{center}\qquad
	\includegraphics[scale=0.7]{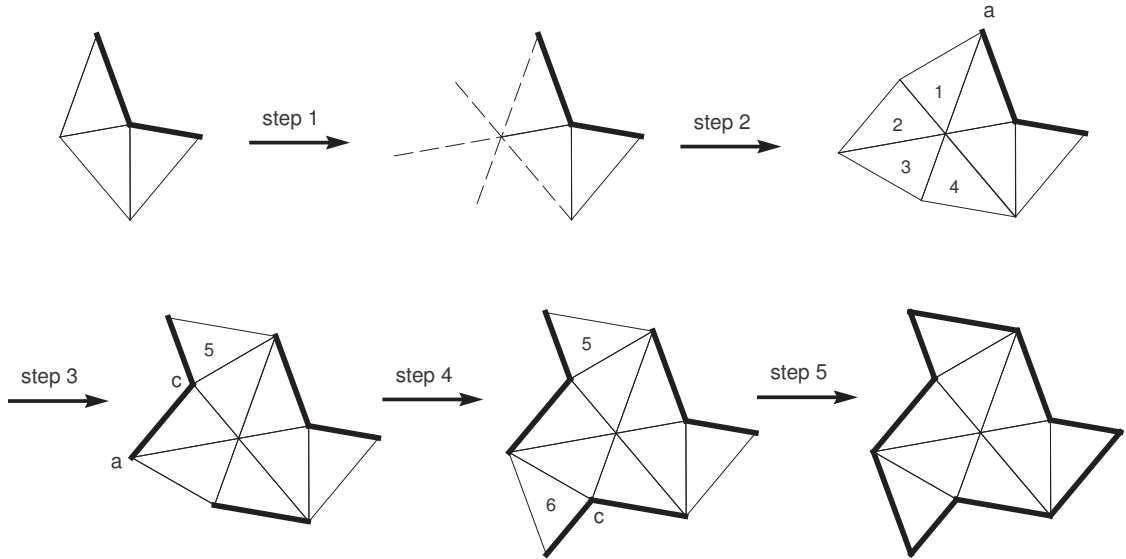}
	\caption{Steps in the construction of Q}
	\label{kesm9_building}
	\end{center}
	\end{figure}

	\textbf{Steps 1 and 2:} The broken lines form an angle of $\frac{k\pi}{9}, k\nmid 9$ with one of the bold lines. 
		Therefore, these sides must be internal. Otherwise, $G_Q$ (the group generated by reflections of the sides of $Q$)
		will be isomorphic to $D_9$, and not isomorphic to $D_3$ as required. Hence, we added triangles number 1, 2, 3, and 4. \quad
	\textbf{Step 3:} According to ($\star$), we added triangle number 5 in order to expand the angle of vertex $a$ to $3\cdot\frac{2\pi}{9}$.
		The new bold lines were added since we could not enlarge the angle  of vertex $c$ (following ($\star$)). \quad
	\textbf{Step 4:} Repeating step 3 - adding triangle number 6 in order to fix the angle of vertex $a$. \quad
	\textbf{Step 5:} In each vertex $a$ we have an angle of $3\cdot\frac{2\pi}{9}$ which we cannot expand according to ($\star$). 
		Therefore, we add the bold lines.

	At that point we have a polygon $Q$ which we cannot enlarge, since all its sides are exterior. According to Propositions 
	\ref{branching} and \ref{commensurable}, $Q$ determines a lattice surface. Therefore, it will not yield an appropriate cover of $M_P$.

	Up to this point we have shown that there is no appropriate cover of the first class (as mentioned in page 12). Now, we 
	will examine the possibility of finding an appropriate cover of the second class.

	According to Lemma \ref{non-periodic_points}, the points corresponding to vertex $b$ are non-periodic. Since the branch points of the cover 
	can arise only from the vertices in $P$, the periodic points in $M_P$, that $M_{\overline{P}}$ can be branched over, are the singular points 
	corresponding to vertices $a$ and $c$.

	($\clubsuit$) Notice that $M_{\overline{P}}$ cannot be branched over the singular point corresponding to vertex $c$. If so, following 
	Proposition \ref{branching}, $\overline{P}$ will have one of the angles $2 \cdot \frac{4\pi}{9}$ or $4 \cdot \frac{4\pi}{9}$.
	In both cases, for any polygon $Q$ tiled by $\overline{P}$, $N_Q=9$ as opposed to the requirement that $N_Q=3$. In particular, according to
	Proposition \ref{branching}, all the angles of the vertices $c$ in $\overline{P}$ must be $\frac{4\pi}{9}$ or $3\cdot \frac{4\pi}{9}$.
	Also, according to Proposition \ref{branching}, since we want $M_{\overline{P}}$ to branch over periodic points in $M_P$, all the 
	vertices in $\overline{P}$, corresponding to the angle $\frac{\pi}{3}$, have to appear with angles $\frac{\pi}{3}$, $\pi$ or $2\pi$.
	
	We will try to construct a suitable $\overline{P}$ for such an appropriate cover. Since we are interested in a polygon $Q$, tiled by 
	$\overline{P}$, such that $M_Q$ is branched over a single non-periodic point, we must have a vertex with angle $\frac{\pi}{3}$ in 
	$\overline{P}$. Hence, we will start the construction of $\overline{P}$ with triangle number 1 with two external sides (as described in 
	Figure \ref{kesm9_overlineP}). The following steps were taken:

	\textbf{Step 1:} The only way to continue is by adding triangle number 2. \quad
	\textbf{Step 2:} According to ($\clubsuit$), each vertex $c$ can appear in $\overline{P}$ with angle $\frac{4\pi}{9}$ or 
		$3\cdot\frac{4\pi}{9}$. Therefore, we add triangle number 3 with an external side. \quad

	For the next step, we will explain why the side of triangle number 2 must be external in $\overline{P}$.
	We have started the construction of $\overline{P}$ with an angle$\frac{\pi}{3}$ in order to multiply this angle to be 
	$\frac{k\pi}{3}$ with $k>1$ and $(k,3)=1$ in $Q$. That is, for having the wanted branching of the cover 
	$\pi:M_Q \to M_{\overline{P}}$. Hence, to find a suitable polygon $Q$, we have to reflect with respect to at least one of the sides 
	of that angle. We cannot reflect with respect to side B (since the angle at vertex $c$ greater than $\pi$). Therefore,
	we must reflect with respect to side $A$.	Recall that all of the angles  in $Q$ must have the form $\frac{k\pi}{3}$. Reflection with 
	respect to the side $A$, multiplies the angle at vertex $a$. If that angle is $k\cdot\frac{2\pi}{9}$ with $k\geq 3$, than according to 
	Proposition \ref{branching}, such a reflection will cause a branching over a singular point in $M_P$. Therefore, we cannot allow 
	expansion of the angle of vertex $a$.

	\textbf{Step 3:} Adding the bold line for an external side in triangle number 2. \quad
	\textbf{Step 4:} Fixing the angle $\frac{2\pi}{3}$ of vertex $b$ to $\pi$ by adding triangle number 4 (following ($\clubsuit$). \quad
	\begin{figure}[h!]
	\begin{center}\qquad
	\includegraphics[scale=0.57]{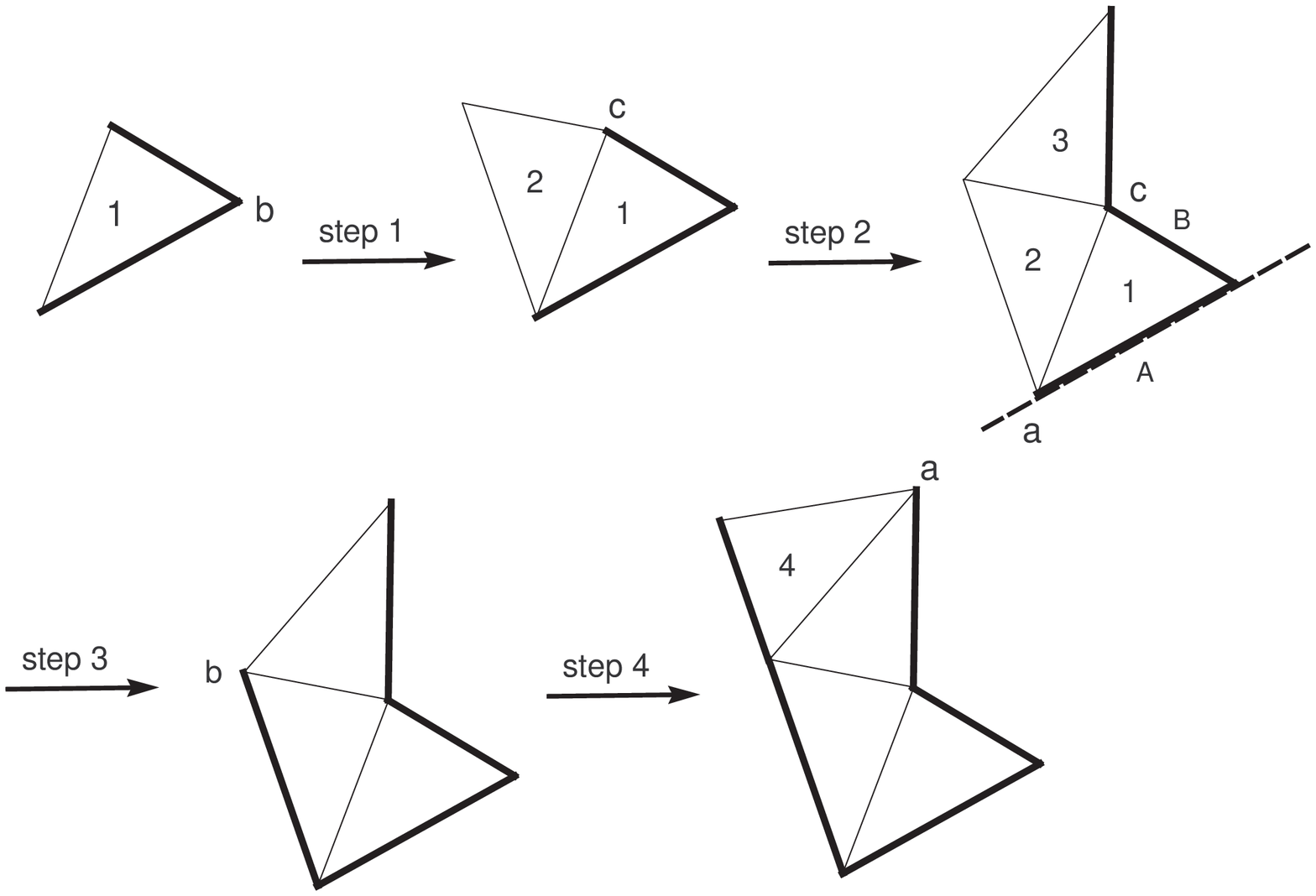}
	\caption{Steps in construction of $\overline{P}$}
	\label{kesm9_overlineP}
	\end{center}
	\end{figure}

	According to Propositions \ref{branching} and \ref{commensurable}), the polygon in the last step has the lattice property. 
	We will show that on one hand, this polygon does not yield an appropriate cover, and on the other hand, we cannot enlarge 
	$\overline{P}$ to get an appropriate cover.

	In the last step, the angle of vertex $a$ is $2\cdot\frac{2\pi}{9}$. Since all the angles in $Q$ must be of the form $\frac{k\pi}{3}$, 
	we need to reflect at the sides of this angle twice. Since the angle of vertex $c$ is greater than $\pi$, we can reflect 
	just once. Therefore, this polygon $\overline{P}$ will not yield an appropriate cover.

	On the other hand, as described in Figure \ref{kesm9_contP}, we cannot enlarge $\overline{P}$. If we add triangle number 5, we must add 
	triangle number 6 as well (according to ($\clubsuit$)). In that case, as we explained before, we cannot reflect in the sides of vertex 
	$a$ twice. Hence, it will not yield an appropriate cover. \begin{figure}[h!]
	\begin{center}\qquad
	\includegraphics[scale=0.61]{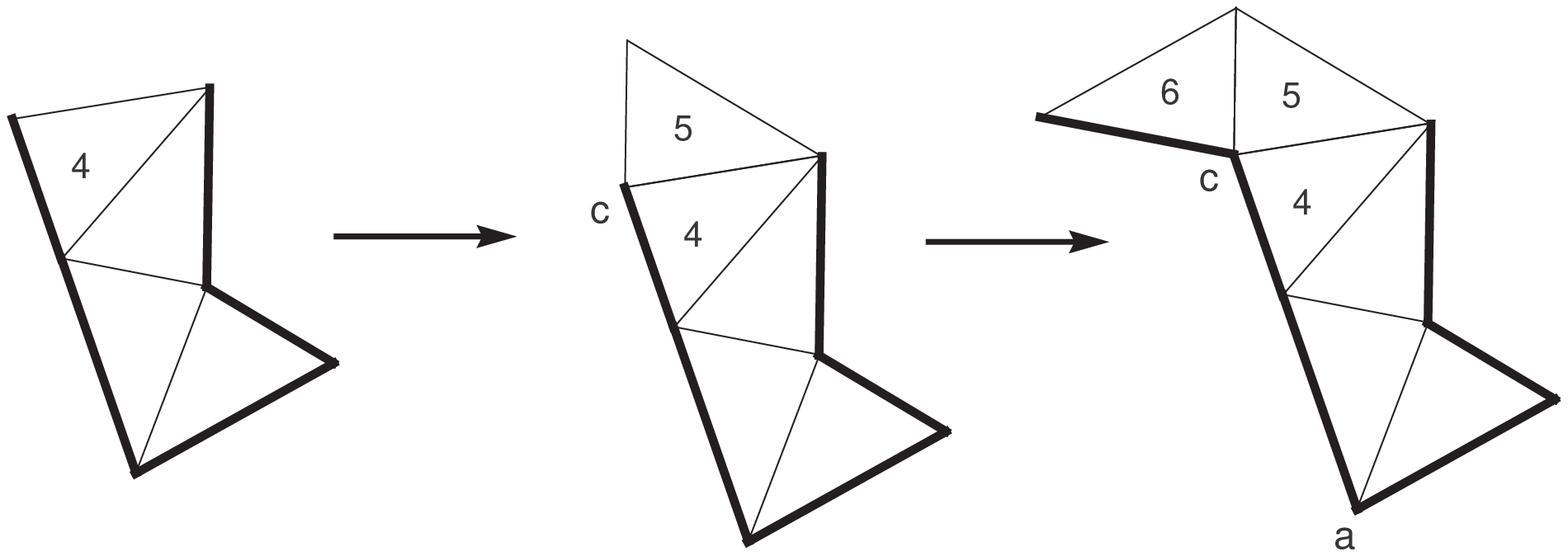}
	\caption{Continuing the construction of $\overline{P}$}
	\label{kesm9_contP}
	\end{center}
	\end{figure}

\end{itemize}

\item \textbf{\boldmath{$P$} is an obtuse triangle with angles
	$\left(\frac{\pi}{2n}, \frac{\pi}{n}, \frac{(2n-3)\pi}{2n}\right)$, $n\geq 4$.}
	\begin{itemize}
	\item [a)] \textbf {If \boldmath{$n$} is even:}
		All the points in $M_P$ (see Figure \ref{ward6}) corresponding to vertices in $P$ are fixed points of the rotation by $\pi$. 
		Therefore by Lemma \ref{involution} these are periodic points. Since we want a single non-periodic branch point, there is no 
		appropriate cover.
		\begin{figure}[h!]
		\begin{center}
		\includegraphics[scale=0.37]{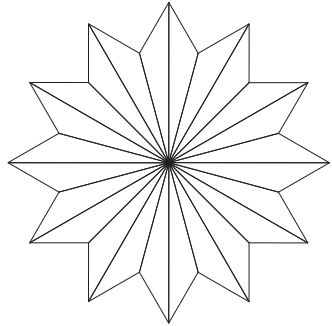}
		\caption{Ward's star for n=6}
		\label{ward6}
		\end{center}
		\end{figure}

	\item [b)] \textbf {If \boldmath{$n$} is odd:}
		Following Remark \ref{observation}, $N_P=2n \Longrightarrow N_Q \mid 2n$. By Lemma \ref{not_appropriate}, if there is an appropriate 
		cover, $N_Q$ must be odd. Hence, since $3\cdot\frac{(2n-3)\pi}{2n} > 2\pi$, all the angles in $Q$ which correspond to the obtuse angle in 
		$P$, must be $2\cdot\frac{(2n-3)\pi}{2n}=\frac{(2n-3)\pi}{n}$. Consequently, $N_Q=n$, and $G_Q\cong D_n$, and each angle
		$2\cdot\frac{(2n-3)\pi}{2n}$ in $Q$ must be delimited by 2 external sides (since we cannot expand it). ($\clubsuit$) 

		According to Lemma \ref{involution}, the center of $M_P$ (see Figure \ref{surfaces_lemma}, surface 5) is a periodic point, 
		since it is a fixed point of the rotation by $\pi$. According to Lemma \ref{non-periodic_points}, the points $c_1$ and 
		$c_2$ (marked in Figure \ref{surfaces_lemma}) are non-periodic points in $M_P$. Therefore, if there is an appropriate cover, 
		it must be branched over one of the points $c_1$ or $c_2$.
		
		($\star$) Notice that, considering Remark \ref{remark_fp_of_G_Q}, since $c_1$ and $c_2$ are not fixed points of the reflections
		with respect to the broken lines in Figure \ref{surfaces_lemma}, these reflections should not be in $G_Q$. Therefore, these sides 
		must be internal in $Q$.

		In the beginning we will examine the first class of appropriate covers (as mentioned in page 12), i.e. a branch cover 
		\text{$\pi:M_Q \to M_P$}, where the branch locus is a single non-periodic point in $M_P$ ($c_1$ or $c_2$).
		
		According to ($\clubsuit$), we will start the construction of $Q$ with an angle $2\cdot\frac{(2n-3)\pi}{2n}$ with two external 
		sides. There are two options for such a beginning, as described in Figure \ref{ward_2_options}.
		\begin{figure}[h!]
		\begin{center}
		\includegraphics[scale=0.65]{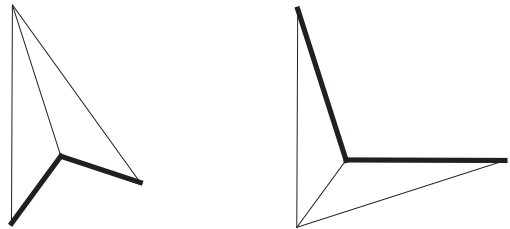}
		\caption{Two options to start the construction of $Q$}
		\label{ward_2_options}			
		\end{center}
		\end{figure}

		According to ($\star$), the right option cannot be contained in $Q$. Therefore, we will start the construction with the left option, 
		under our requirements, step by step as described in Figure \ref{ward5_stages_Q}.

		\textbf{Step 1:} According to Propositions \ref{branching} and \ref{commensurable}, the polygon we started with has the lattice 
			property. Therefore, we have to enlarge it. Without loss of generality we add triangle number 1. \quad
		\textbf{Step 2:} Following ($\star$), the broken line indicates a side which must be internal in $Q$. Therefore, we add triangle 
			number 2. \quad
		At that point, according to Proposition \ref{branching}, the angles of vertices $a$ and $c_1$ will lead to a branching over
		the points corresponding to the center of $M_P$ and $u$. Since we are interested in branching over a single non-periodic point, 
		we must continue the construction. Without loss of generality we add triangle number 3. Again, following ($\star$), the broken 
		line indicates a side which must be internal in $Q$. Therefore, we add triangle number 4.
		Now, following Proposition \ref{branching}, we have a branching over two different points, $c_1$ and $c_2$ in $M_P$.
		Since both vertices $u$ and $v$ are delimited by 2 external sides, we cannot fix the angle to prevent the branching over one of these
		points. Hence, we cannot find an appropriate cover of the first class.
		\begin{figure}[h!]
		\begin{center}
		\includegraphics[scale=0.7]{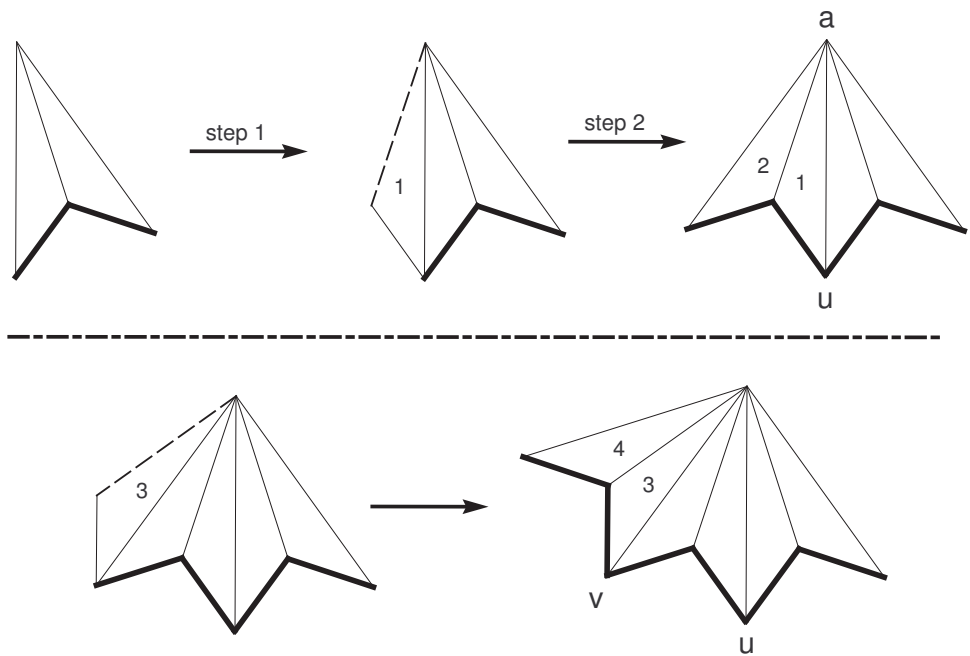}
		\caption{Steps of the construction of $Q$}
		\label{ward5_stages_Q}			
		\end{center}
		\end{figure}
		
		It remains to examine the second class of appropriate covers. We need to construct a polygon $\overline{P}$, such that 
		$M_{\overline{P}}$ is branched only over periodic points in $M_P$, i.e. the center of $M_P$ or the singular point. 
		According to Proposition \ref{branching} the cover will be branched over the singular point if and only if $\overline{P}$ will have 
		an angle $k\cdot\frac{2n-3}{2n}\pi$ with $k\nmid{2n}$. Since $3\cdot\frac{2n-3}{2n}\pi > 2\pi$, there will not be a branching
		over the singular point. Therefore, for such a polygon $\overline{P}$, $M_{\overline{P}}$ must be branched over the center of $M_P$, 
		corresponding to the angle $\frac{\pi}{2n}$ in $P$. Hence, according to Proposition \ref{branching}, $\overline{P}$ must have an angle
		$k\cdot\frac{\pi}{2n}$ with $k\nmid 2n$. Such an angle cannot appear in $\overline{P}$ without a branching over a non-periodic point, 
		as described in Figure \ref{ward5_overlineP}:
		If $k\geq 3$, then $\overline{P}$ contains the left polygon in Figure \ref{ward5_overlineP} (the sides are marked as external in 
		the figure because of formula $\clubsuit$). According to Proposition \ref{branching}, since the angle of vertex $u$ is 
		$2\cdot\frac{\pi}{n}$ with odd $n$, there will be a branching over a non-periodic point. Hence, we have to fix that angle. Since one 
		of the sides of this angle is external, we can fix it only by adding triangle number 1. In that case, we have 2 external sides with 
		an angle greater than $\pi$. Hence, each polygon $Q$ tiled by $\overline{P}$ will have these external sides as well, which 
		contradicts ($\star$). Therefore, the other side of the angle of vertex $u$ must be external as well. In that case, the cover 
		$\pi:M_{\overline{P}} \to M_Q$ is branched over a non-periodic point. Consequently, there is not appropriate cover of the second 
		class neither.
		\begin{figure}[h!]
		\begin{center}
		\includegraphics[scale=0.6]{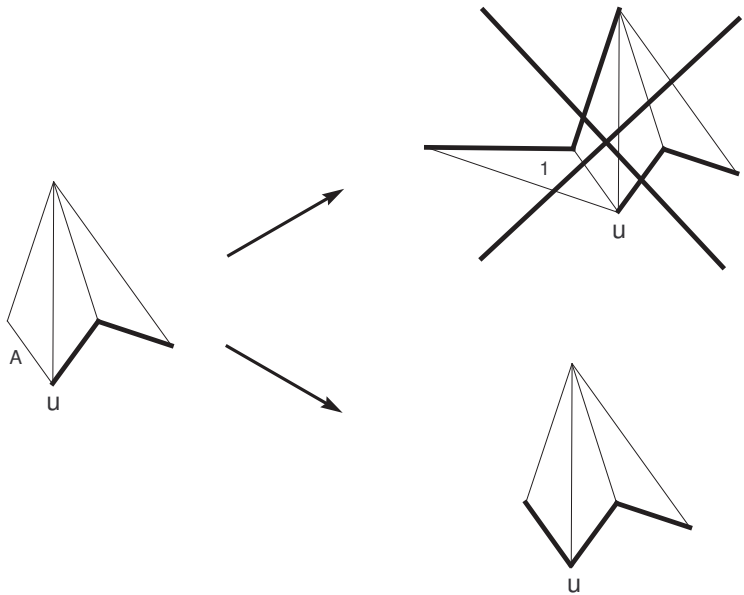}
		\caption{Steps of building $Q$}
		\label{ward5_overlineP}			
		\end{center}
		\end{figure}

	\end{itemize}

\item \textbf{\boldmath{$P$} is the obtuse triangle with angles 
	$\left(\frac{\pi}{12}, \frac{\pi}{3}, \frac{7\pi}{12}\right)$.}\\
	Here $N_P=12$. Remark \ref{observation} implies that 
	$N_Q \in \{2,3,4,6,12\}$. According to Corollary \ref{not_appropriate}, if there is an appropriate cover $\pi:M_Q \to M_P$, then $N_Q$ must 
	be odd. Consequently, $N_Q$ must be 3. That means that all the angles of vertices in $Q$ are of the form $\frac{k}{3}\pi$. 
	Hence, the angle $\frac{7}{12}\pi$ in $P$ has to be multiplied by $4n$.	Since $4\!\cdot\! \frac{7}{12}\pi > 2\pi$, there is 
	no appropriate cover. 

\item \textbf{\boldmath{$P$} is a L-shaped polygon:}
	All the points in $M_P$ corresponding to vertices in $P$ are fixed points of the rotation by $\pi$. Hence, by Lemma \ref{involution}, 
	these are all periodic points. Since we looking for a cover with a single non-periodic branch point, there is no appropriate cover.

\item \textbf{Bouw and M\"{o}ller examples} (See \cite{BM} for a description):
	\begin{itemize}
		\item 4-gon with angles $\left(\frac{\pi}{n},\frac{\pi}{n},\frac{\pi}{2n},\frac{(4n-5)\pi}{2n}\right)$
				for $n\geq 7$ and odd.
		\item 4-gon with angles $\left(\frac{\pi}{2},\frac{\pi}{n},\frac{\pi}{n},\frac{(3n-4)\pi}{2n}\right)$ 
				for $n\geq 5$ and odd.
	\end{itemize}
	For each of the polygons above, $N_P=2n$. Since $\frac{(4n-5)\pi}{2n}$ and $\frac{(3n-4)\pi}{2n}$ are even
	angles greater than $\pi$, any polygon $Q$ tiled by $P$ (one of the polygons above) must have these angles as well. 
	Therefore, $N_Q$ must be even. Corollary \ref{not_appropriate} implies that there is no appropriate cover for those examples.

\item \textbf{\boldmath{$P$} is a square-tiled polygon:}\\
	Consider the following theorem by Gutkin and Judge \cite{GJ00}: \textit{A surface $M$ is tiled by parallelograms if and only if 
	\ $\Gamma(M)$ is arithmetic}. Since any arithmetic group is a lattice, this theorem proves that all the square-tiled surfaces 
	are lattice surfaces. If $M$ is a square-tiled surface, then every surface $\widetilde{M}$, 
	which covers $M$, is also a square-tiled surface, hence a lattice surface. Therefore, we cannot have an appropriate cover in this case.
\end{enumerate}

\end{document}